\documentclass{amsart}
\usepackage{enumitem,sidenotes,float}

\theoremstyle{definition}
\newtheorem{theorem}{Theorem}

\newtheorem{proposition}{Proposition}

\newtheorem{remark}{Remark}

\allowdisplaybreaks
\newcommand{\norm}[1]{\|#1\|}
\newcommand{\ip}[2]{\langle #1,#2\rangle}
\newcommand{\abs}[1]{|#1|}
\newcommand{\qqquad}{\qquad\quad}
\numberwithin{equation}{section}
\newcommand{\opi}{\overline{\pi}}

\title{Differential operators
  on {B}ergman space on bounded symmetric domains}
\author{Jens Gerlach Christensen}
\address{ Department of Mathematics, Colgate University} 
\email{jchristensen@colgate.edu}
\urladdr{http://www.math.colgate.edu/~jchristensen}

\author{Christopher Benjamin Deng}
\address{ Department of Mathematics, Cornell University}
\email{cbd62@cornell.edu}
\urladdr{https://math.cornell.edu/christopher-deng}

\begin{document}

\maketitle

\begin{abstract}
	We classify self-adjoint first-order differential operators
	on weighted Bergman spaces on the $N$-dimensional unit ball $\mathbb{B}^N$ 
	and $\mathbb{D}^2$ of $2\times2$ complex matrices satisfying $I-ZZ^*>0$.
	Our main tools are the discrete series representations
	of $\mathrm{SU}(N,1)$ and $\mathrm{SU}(2,2)$. 
	We believe that our observations extend to general bounded symmetric domains.
  \end{abstract}

\section{Introduction}
In a recent paper \cite{Deng2022}, the authors showed that all first order self-adjoint differential
operators on weighted Bergman space on the unit disc come from the holomorphic discrete series representation.
The purpose of the present paper is to test if this result generalizes to
to higher dimensional and higher rank bounded symmetric domains. In particular, we show that
this is the case for the Bergman space on the (rank 1) unit ball $\mathbb{B}^n$ in $\mathbb{C}^n$
and the rank 2 space $D$ consisting of $2\times 2$ complex matrices with operator norm less than 1.
We believe this gives us enough reason to claim that our observations extend to general bounded symmetric domains.
The proof of this claim eludes us at this stage. It further begs the question if we can give a Lie/representation
theoretic classification of higher order self-adjoint differential operators on Bergman spaces.

\section{Background and statement of main result}

\subsection{Lie groups and representations}
Let $G$ be a Lie group with Lie algebra $\mathfrak{g}$ and let $\pi$ be a representation of the group on a Hilbert space $H$.
The space of smooth vectors $H^\infty_\pi$ is the (dense in $H$) collection of vectors $v$ for which
$x\mapsto \pi(x)v$ is a smooth mapping from $G$ to $H$. For every $X\in\mathfrak{g}$ we get an operator $\pi(X)$
with domain $H_\pi^\infty$ defined by
$$
\pi(X)v = \frac{d}{dt}\Big|_{t=0} \pi(\exp(tX))v.
$$
This operator is densely defined and skew-symmetric, and therefore it is closable. The closure
will be denoted $\opi(X)$ and it is skew-adjoint (see \cite{Segal1951}).

\subsection{Bounded symmetric domains of type AIII}

In this paper we will restrict ourselves to bounded
symmetric domains $D$ consisting of
complex $N\times M$ matrices $Z$ for which
$I-Z^*Z>0$. The group $G=SU(N,M)$
consists of block matrices 
$x=\begin{bmatrix}
  a & b \\
  c & d
\end{bmatrix}$ with determinant $1$, $a$ is $N\times N$,
$b$ is $N\times M$, $c$ is $M\times N$ and $d$ is $M\times M$ and
which satisfy $a^*a-c^*c=I_{N\times N}$, $a^*b=c^*c$
and $b^*b-d^*d=-I_{M\times M}$.
The group $G$ acts transitively on the domain $D$ by the action
$$
x\cdot Z = (aZ+b)(cZ+d)^{-1}.
$$
The subgroup $K$ that fixes the origin $o$
consists of matrices
$\begin{bmatrix}
   a & 0 \\
   0 & d
 \end{bmatrix},
 $
 in $S(U(N)\times U(M))$
 and the domain $D$ can be identified with the homogenous space $G/K$.
 The complex Jacobian $J(x,Z)$ of the mapping $Z\mapsto x\cdot Z$ is given by
 $$
 J(x,Z) = \det(cZ+d)^{-(N+M)},
 $$
 By the chain rule $J(xy,Z) = J(x,y\cdot Z)J(y,Z)$.

\subsection{Bounded symmetric domains and Bergman spaces}

Define the function $h(Z)=\det(I-Z^*Z)$ which is strictly positive
on $D$ and satisfies
$h(x\cdot Z) = |J(x,Z)|^{2/(N+M)}h(Z)$.
The weighted Bergman space
$A^p_\xi(D)$ is the space of holomorphic functions on $D$
for which
$$
\int_D |f(Z)|^ph(Z)^{\xi}\,dZ <\infty
$$
and it contains constant functions (is non-zero) when $\xi>-1$.
It forms a Banach space when equipped with the norm
$$
\| f\|_{p,\xi} = \left( \int_D |f(Z)|^ph(Z)^{\xi}\,dZ\right)^{1/p}.
$$

When $p=2$ the space $A^2_\xi(D)$ is a Hilbert space with inner
product
$$
\langle f,g\rangle = \int_D f(Z)\overline{g(Z)}h(Z)^\xi \,dZ
$$
and the norm in this case will be denoted simply $\| f\|_\xi$.
For $\xi>-1$ the holomorphic discrete series represention $\pi_\xi$ of
$G$ on $A^2_\xi(D)$ is given by
$$
\pi_\xi(x)f(Z) = \det(cZ+d)^{-(\xi+N+M)} f(x^{-1}\cdot Z)
$$
when $x^{-1}=\begin{bmatrix}a & b \\ c&d
\end{bmatrix}$.
This (projective) representation is unitary and irreducible.
Notice that it is only a representation of $G$ if $\xi$ is an integer,
but that it defines a projective representation of the universal cover
of $G$ in other cases. In this paper, we will only concern ourselves
with differential operators arising from the Lie algebra $\mathfrak{g}$,
so this distinction between the groups is irrelevant.

\subsection{First order differential operators}

For multi-indices $\alpha$ and functions $f_\alpha \in A^2_\xi(D)$
we define a first order differential
operator $L:A^2_\xi(D)\to A^2_\xi(D)$ to be of the form
$$
L = f_0 + \sum_{|\alpha|=1} f_\alpha \partial^\alpha.
$$
This operator has domain $\mathcal{D}(L)$ consisting of polynomials
in $N$ variable ($N$ is the dimension of the ambient $\mathbb{C}^N$)
and it is thus densely defined. We will now show that if $L$ is symmetric,
then the functions $f_\alpha$ are polynomials of degree $|\alpha|+1$ or less.
This relies on the fact that homogeneous polynomials of differing 
degree are orthogonal in $A^2_\xi(D)$ (see \cite{Hua1963} and the
fact that $D$ is circular \cite{Koranyi1965}).

\begin{proposition}\label{prop1}
  Let $L = f_0 + \sum_{|\alpha|=1} f_\alpha \partial^\alpha$ be a first order
  differential operator on $A^2_\xi(D)$
  with polynomials as domain. If $L$ is symmetric,
  then $f_\alpha$ is a polynomial
  of degree less than or equal to $|\alpha|+1$.
\end{proposition}

\begin{proof}
  Let $p=1$, and let $q$ be a monomial of degree two or higher.
  Then $Lp=f_0$ so $\langle Lp,q\rangle = \langle f_0,q\rangle$.
  On the other hand $Lq$ is a polynomial of degree one or higher,
  so $\langle p,Lq\rangle = 0$. Since the monomials form a basis
  for $A^2_\xi(D)$ (not necessarily orthogonal), and $\langle f_0,q\rangle=0$
  for any monomial of degree greater than or equal to two,
  $f_0$ has degree less than or equal to one.
  
  Repeating the argument with $p=z^\alpha$ of degree one
  and $q$ of degree three or higher,
  we can conclude that $f_\alpha$ has degree two or less.
  We get that $Lp=f_0z^\alpha+f_\alpha$, and since $f_0z^\alpha$ has degree
  two or less, then $\langle Lp,q \rangle = \langle f_\alpha,q\rangle$.
  Moreover, $Lq$ has degree two or higher, so $\langle p,Lq\rangle =0$,
  and this concludes the argument.
\end{proof}

\subsection{Main results}
We finally have sufficient background to state our main results.
Moving forward, the domain $D$ will be either the
unit ball $\mathbb{B}^N=\{ z\in\mathbb{C}^N : |z|<1\}$ or
the domain $\mathbb{D}^2$ of $2\times 2$ complex matrices
$Z$ for which $I-Z^*Z>0$. Let
$A^2_\xi(D)$ be the corresponding Bergman space.
$G$ will be the group $SU(N,1)$ in the case of $\mathbb{B}^N$
and $SU(2,2)$ in the case of $\mathbb{D}^2$. Lastly, $\pi_\xi$ will be the
corresponding discrete series (projective) representation
of the group $G$ whose Lie algebra will be denoted $\mathfrak{g}$.
We seek to prove the following:

\begin{theorem}\label{mainthm}
  For the two domains $\mathbb{B}^N$ and $\mathbb{D}^2$,
  the closure of any first
  order self-adjoint differential
  operator on $A^2_\xi(D)$ is equal to $c+i\overline{\pi}_\xi(X)$
  for some real $c$ and some $X\in\mathfrak{g}$.
\end{theorem}

In the case of the unit ball, we also show that
\begin{theorem}\label{euler-isometry}
  The operator $\widetilde{L}:=\sum_{j=1}^Nz_j\frac{\partial}{\partial z_j}+c$ with $c$ not being 0 or a negative integer 
  extends to a linear homeomorphism between $\mathcal{A}_\xi^2(\mathbb{B}^N)$ and $\mathcal{A}_{\xi+2}^2(\mathbb{B}^N)$.
\end{theorem}

\section{Verification of results}
We will use the fact that monomials of differing homogeneous degrees
are orthogonal to derive the form of the first order differential operator
$L = f_0+\sum_{|\alpha|=1} f_\alpha \partial^\alpha$. In fact, we will only
need monomials of order 2 or less to show that $L$ agrees
with $c+i\pi(X)$ for some $c\in\mathbb{R}$ and $X\in\mathfrak{g}$.
We know all such operators are symmetric on $H_\pi^\infty$ which contains
polynomials, and therefore the symmetry of $L$ restricted to
polynomials follows automatically. This saves a lot of time
over the approach utilized in \cite{Villone1970} for
the case of the unit disk.

\subsection{General observations}

Later we will need the following characterization of
the smooth vectors for $\pi_\xi$ due to \cite{Chebli2004}.
Notice that any holomorphic function $f$ can be decomposed as
$f = \sum_{k=0}^\infty f_k$ where $f_k$ is a homogeneous polynomial of degree $k$.
\begin{theorem}
  The smooth vectors for $\pi_\xi$
  are the holomorphic functions $f=\sum_{k=0}^\infty f_k$ for which
  for any $n$ there is a constant $C_n>0$ such that 
  the norm   $\| f_k\|_\xi < C_n(1+k)^{-n}$ for all $k=0,1,2,\dots$.
\end{theorem}

This classification can be used to verify the following result, which
we have not found a reference for, and we therefore include a proof.

\begin{proposition}
  The polynomials are a core (essential domain) for $\pi_\xi(X)$.
\end{proposition}

\begin{proof}
  We need to verify that if $A$ is $\pi_\xi(X)$ restricted to the polynomials,
  then $\overline{A}$ is the same as $\overline{\pi}_\xi(X)$.
  This is the same as showing that if $f\in H_{\pi_\xi}^\infty$, then
  there is a sequence of polynomials $p_n$ converging to $f$
  for which $Ap_n=\pi(X)p_n$ converges of $\pi_\xi(X)f$.
  
  First note that due to the classification of smooth vectors,
  if $f=\sum_k f_k$, then the series is absolutely convergent in
  $H$ ($\sum_{k=0}^\infty \| f_k\| \leq C \sum_{k=0}^\infty (1+k)^{-2}$)
  and therefore it converges in $H$.
  This means that $p_n = \sum_{k=0}^n f_k$ converges to $f$.
  Now, $f-p_n$ is also a smooth vector, and $\pi_\xi(X)(f-p_n)$ will
  be a smooth vector. 
  
  Notice that
  $$\pi_\xi(X)f(Z) =
  \frac{d}{dt}\Big|_{t=0} \pi_\xi(e^{tX})f(Z)
  =\frac{d}{dt}\Big|_{t=0} \det(c(t)Z+d(t))^{-(\xi+N+M)}f(e^{-tX}\cdot Z)
  $$
  where $e^{-tX} =
  \begin{bmatrix}
    a(t) & b(t) \\ c(t) & d(t)
  \end{bmatrix}
  $.
  The derivative of $\det(cZ+d)^{-(\xi+N+M)}$ at $t=0$ is a polynomial
  in the entries of $Z$, and therefore the product rule show that
  $\pi_\xi(X)$ will be of the form
  $$
  \pi_\xi(X) = p_0 + \sum_{|\alpha|=1} p_\alpha\partial^\alpha
  $$
  where $p_0$ and $p_\alpha$ are polynomials in the entries of $Z$
  and $\partial^\alpha$ is a partial derivative
  (of order 1) in the coordinates of $Z$.

  Therefore $\pi_\xi(X)(f-p_n)$ cannot contain any power less than $n$,
  i.e. $\pi_\xi(X)(f-p_n) = \sum_{k\geq n} g_k$ where $g_k$ is a polynomial of
  homogeneous degree $k$. The series $\sum_{k\geq n} g_k$ converges absolutely
  (same argument as before), so $\pi_\xi(X)(f-p_n)$ converges to $0$
  which finishes the proof.
\end{proof}

\subsection{The case of the Unit Ball}

	We will denote tuples $z\in\mathbb{C}^N$ as $z=(z_1,\cdots,z_N)$.
	Let $H(\mathbb{B}^N)$ be the space of all holomorphic functions on the unit ball
	\begin{align}
		\mathbb{B}^N
		&=\{z\in\mathbb{C}^N\,\mid\, |z|^2:=|z_1|^2+\cdots+|z_N|^2<1\}. \nonumber
	\end{align}
	As mentioned in \cite{Christensen2017,Zhu2000}, we can identify $\mathbb{B}^N$ with the unit ball in $\mathbb{R}^{2N}$, and thus equip the measure $\mathrm{d}\nu=2Nr^{2N-1}\,\mathrm{d}r\mathrm{d}\pi_N$, where $\mathrm{d}\pi_N$ is the rotation-invariant surface measure on the sphere $\mathbb{S}^{2N-1}\subseteq\mathbb{R}^{2N}$ normalized by $\pi_N(\mathbb{S}^{2N-1})=1$.
	Define the weighted measure 
	\begin{align}
		\mathrm{d}\nu_\xi(z):=\frac{\Gamma(N+\xi+1)}{N!\Gamma(\xi+1)}(1-|z|^2)^\xi\,\mathrm{d}\nu(z). \nonumber
	\end{align}
	For $\xi>-1$, this is a probability measure and therefore
	the weighted Bergman space defined by
	\begin{align}
	  \mathcal{A}_\xi^p(\mathbb{B}^N):=\left\{ f\in H(\mathbb{B}^N)\,\bigg|\,\|f\|_{\mathcal{A}_\xi^p}:=\left(\int_{\mathbb{D}}|f(z)|^p\,\mathrm{d}\nu_\xi(z)\right)^{1/p} <\infty\right\} \nonumber
	\end{align}
	is non-trivial for $\xi>-1$.
	
	In this paper we will mostly focus on the Hilbert space
	$\mathcal{A}^2_\xi(\mathbb{B}^N)$ with inner product
	$$
	\ip{f}{g}_{\xi} := \int_{\mathbb{B}^N} f(z)\overline{g(z)}\,\mathrm{d}\nu_\xi(z).
	$$
	The monomials $z^n$ form an orthogonal basis, where $\abs{n}=\sum_{k=1}^Nn_k$ and $n!=\prod_{k=1}^Nn_k!$ for $n=(n_1,\cdots,n_N)\in\mathbb{N}_0^N$.
	We also denote the norm on this space by $\|\cdot\|_\xi$, and make note that 
	\begin{align}
		\|z^n\|_\xi^2
		&=\frac{\Gamma(N+\xi+1)n!}{\Gamma(N+\xi+1+\abs{n})}. \nonumber
	\end{align}

	\begin{proposition} \label{FODO}
		Let $f_\alpha\in\mathcal{A}_\xi^2(\mathbb{B}^N)$ be denoted $f_\alpha(z)=\sum_{\beta\in\mathbb{N}_0^N}a_\alpha^\beta z^\beta$ where $\alpha\in\mathbb{N}^N$.
		The operator $L=f_0+\sum_{k=1}^N f_{e_k} \frac{\partial}{\partial z_k}$ with domain $\mathcal{D}(L)=P(\mathbb{B}^N)$ is symmetric if and only if 
		\begin{align}
				f_0&=a_0^0+(N+\xi+1)\sum_{j=1}^N \overline{a_{e_j}^0}z_j, \label{form1} \\
				f_{e_k}&=a_{e_k}^0+\sum_{j=1}^N (a_{e_k}^{e_j} z_j+\overline{a_{e_j}^0}z_jz_k) \label{form2}
		\end{align}
		where the coefficients satisfy 
		\begin{align}
                  a_0^0,a_{e_k}^{e_k}\in\mathbb{R}, \qqquad a_0^{e_k}=(N+\xi+1)\overline{a_{e_k}^0}, \qqquad a_{e_k}^{e_j}=\overline{a_{e_j}^{e_k}}. \label{relations1}
		\end{align}
	\end{proposition}
		\begin{proof}
			Assume that $L$ is symmetric.
			It is enough to work with the monomials, since they form an orthogonal basis for $P(\mathbb{B}^N)$.
			For any $n\in\mathbb{N}_0^N$,
			\begin{align}
				Lz^n
				&=\left(\sum_{\beta\in\mathbb{N}_0^N}a_0^\beta z^\beta+\sum_{\abs{\alpha}=1}\sum_{\beta\in\mathbb{N}_0^N}a_\alpha^\beta z^\beta \partial^\alpha\right)z^n \nonumber \\
				&=\sum_{\beta\in\mathbb{N}_0^N}a_0^\beta z^{\beta+n}+\sum_{\abs{\alpha}=1}\sum_{\beta\in\mathbb{N}_0^N}a_\alpha^\beta(\alpha\cdot n)z^{\beta+n-\alpha}. \nonumber
			\end{align}
			Then, we have for all $m\in\mathbb{N}_0^N$ that
			\begin{align}
				\ip{Lz^n}{z^m}_\xi
				&=\sum_{\beta\in\mathbb{N}_0^N}a_0^\beta\ip{z^{\beta+n}}{z^m}_\xi+\sum_{\abs{\alpha}=1}\sum_{\beta\in\mathbb{N}_0^N}a_\alpha^\beta(\alpha\cdot n)\ip{z^{\beta+n-\alpha}}{z^m}_\xi \nonumber
			\end{align}
			and that
			\begin{align}
				\ip{z^n}{Lz^m}_\xi 
				&=\overline{\ip{Lz^m}{z^n}}_\xi \nonumber \\
				&=\sum_{\beta\in\mathbb{N}_0^N}\overline{a_0^\beta\ip{z^{\beta+m}}{z^n}_\xi}+\sum_{\abs{\alpha}=1}\sum_{\beta\in\mathbb{N}_0^N}\overline{a_\alpha^\beta(\alpha\cdot m)\ip{z^{\beta+m-\alpha}}{z^n}_\xi}. \nonumber
			\end{align}
			By assumption, the two expressions must be equal, so using orthogonality,
			\begin{align}
				\left(a_0^{m-n}+\sum_{\abs{\alpha}=1}a_\alpha^{m+\alpha-n}(\alpha\cdot n)\right)\|z^m\|_\xi^2
				&=\left(\overline{a_0^{n-m}}+\sum_{\abs{\alpha}=1}\overline{a_\alpha^{n+\alpha-m}}(\alpha\cdot m)\right)\|z^n\|_\xi^2 \nonumber 
			\end{align}
			where terms with indices that have negative components are set to 0.
			Notice that
			\begin{align}
				f_0
					=Lz^0
					=\sum_{\beta\in\mathbb{N}_0^N}a_0^\beta z^\beta 
				&=\sum_{\beta\in\mathbb{N}_0^N}\left(\overline{a_0^{-\beta}}+\sum_{\abs{\alpha}=1}\overline{a_\alpha^{\alpha-\beta}}(\alpha\cdot\beta)\right)\frac{\|1\|_\xi^2}{\|z^\beta\|_\xi^2}z^\beta \nonumber \\
				&=\overline{a_0^0}+\sum_{\abs{\beta}=1}\overline{a_\beta^0}\frac{\|1\|_\xi^2}{\|z^\beta\|_\xi^2}z^\beta 
					=\overline{a_0^0}+(N+\xi+1)\sum_{\abs{\beta}=1}\overline{a_\beta^0}z^\beta. \nonumber
			\end{align}
			Similarly, we can see that for $\gamma=(0,\cdots,0,1,0,\cdots,0)$,
			\begin{align}
				f_\gamma
				&=(L-f_0)z^\gamma \nonumber \\
				&=\sum_{\abs{\alpha}=1}\sum_{\beta\in\mathbb{N}_0^N}a_\alpha^\beta(\alpha\cdot\gamma)z^{\beta+\gamma-\alpha}\nonumber \\
				&=\sum_{\beta\in\mathbb{N}_0^N}\left[\left(\overline{a_0^{\gamma-\beta}}+\sum_{\abs{\alpha}=1}\overline{a_\alpha^{\gamma+\alpha-\beta}}(\alpha\cdot\beta)\right)\frac{\|z^\gamma\|_\xi^2}{\|z^\beta\|_\xi^2}-\alpha_0^{\beta-\gamma}\right]z^\beta \nonumber \\
				&=\overline{a_0^\gamma}\frac{\|z^\gamma\|_\xi^2}{\|1\|_\xi^2}+(\overline{a_0^0}+\overline{a_\gamma^\gamma}-a_0^0)z^\gamma+\sum_{\substack{\abs{\beta}=1 \\ \beta\neq\gamma}}\overline{a_\beta^\gamma}z^\beta \nonumber \\
				&\qqquad+\left(2\overline{a_\gamma^0}\frac{\|z^\gamma\|_\xi^2}{\|z^{2\gamma}\|_\xi^2}-a_0^\gamma\right)z^{2\gamma}+\sum_{\substack{\abs{\beta}=2 \\ \beta\neq2\gamma}}\left(\overline{a^0_{\beta-\gamma}}\frac{\|z^\gamma\|_\xi^2}{\|z^\beta\|_\xi^2}-a_0^{\beta-\gamma}\right)z^\beta. \nonumber \\
				&=\overline{a_0^\gamma}\frac{\|z^\gamma\|_\xi^2}{\|1\|_\xi^2}+(\overline{a_0^0}+\overline{a_\gamma^\gamma}-a_0^0)z^\gamma+\sum_{\substack{\abs{\beta}=1 \\ \beta\neq\gamma}}\overline{a_\beta^\gamma}z^\beta \nonumber \\
				&\qqquad+\left(2\overline{a_\gamma^0}\frac{\|z^\gamma\|_\xi^2}{\|z^{2\gamma}\|_\xi^2}-a_0^\gamma\right)z^{2\gamma}+\sum_{\substack{\abs{\beta}=1 \\ \beta\neq\gamma}}\left(\overline{a^0_\beta}\frac{\|z^\gamma\|_\xi^2}{\|z^{\beta+\gamma}\|_\xi^2}-a_0^\beta\right)z^{\beta+\gamma} \nonumber \\
				&=\frac{\overline{a_0^\gamma}}{N+\xi+1}+(\overline{a_0^0}+\overline{a_\gamma^\gamma}-a_0^0)z^\gamma \nonumber \\
				&\qqquad+\sum_{\substack{\abs{\beta}=1 \\ \beta\neq\gamma}}\overline{a_\beta^\gamma}z^\beta+\sum_{\abs{\beta}=1}\left(\overline{a_\beta^0}\frac{N+\xi+2}{N+\xi+1}-a_0^\beta\right)z^{\beta+\gamma}. \nonumber
			\end{align}
			Since $L$ is assumed to be symmetric,
			\begin{align}
				\begin{cases}
					a_0^0=\overline{a_0^0}\in\mathbb{R} & \text{from }(\abs{n},\abs{m})=(0,0), \\
					a_0^m=(N+\xi+1)\overline{a_m^0} & \text{from }(\abs{n},\abs{m})=(0,1),(1,0), \\
					a_n^m=\overline{a_m^n} & \text{from }(\abs{n},\abs{m})=(1,1),m\neq n, \\
					a_m^m=\overline{a_m^m}\in\mathbb{R} & \text{from }(\abs{n},\abs{m})=(1,1),m=n,
				\end{cases} \nonumber
			\end{align}
			which gives us the equality
			\begin{align}
				f_\gamma
				&=a_\gamma^0+\sum_{\abs{\beta}=1}(a_\gamma^\beta z^\beta+\overline{a_\beta^0}z^{\beta+\gamma}) \nonumber
			\end{align}
			as desired.
	
			Assume $L$ satisfies (\ref{form1}, \ref{form2}, \ref{relations1}), so
			\begin{align}
				L&=\left(a_0^0+(N+\xi+1)\sum_{\abs{\beta}=1}\overline{a_\beta^0}z^\beta\right) \nonumber \\
				&\qqquad+\sum_{\abs{\alpha}=1}\left(a_\alpha^0+\sum_{\abs{\beta}=1}(a_\alpha^\beta z^\beta+\overline{a_\beta^0}z^{\beta+\alpha})\right)\partial^\alpha \nonumber
			\end{align}
			where the coefficients satisfy the relations
			\begin{align}
				a_0^0,a_\alpha^\alpha\in\mathbb{R}, \qqquad a_0^\alpha=(N+\xi+1)\overline{a_\alpha^0}, \qqquad a_\alpha^\beta=\overline{a_\beta^\alpha} \nonumber
			\end{align}
			where $\abs{\alpha},\abs{\beta}=1$ and $\alpha\neq\beta$.
			For any $n\in\mathbb{N}_0^N$,
			\begin{align}
				Lz^n 
				&=\left(a_0^0z^n+(N+\xi+1)\sum_{\abs{\beta}=1}\overline{a_\beta^0}z^{\beta+n}\right) \nonumber \\
				&\qqquad+\sum_{\abs{\alpha}=1}\left(a_\alpha^0 z^{n-\alpha}+\sum_{\abs{\beta}=1}(a_\alpha^\beta z^{\beta+n-\alpha}+\overline{a_\beta^0}z^{\beta+n})\right)(\alpha\cdot n). \nonumber 
			\end{align}
			By orthogonality of the $z^n$ and conjugate symmetry of inner products,
			\begin{align}
                          \ip{Lz^n}{z^n}_\xi
                          &=\left(a_0^0+\sum_{\abs{\alpha}=1}a_\alpha^\alpha(\alpha\cdot n)\right)\|z^n\|_\xi^2
                            =\ip{z^n}{Lz^n}_\xi, \nonumber \\
                          \ip{Lz^n}{z^{\gamma+n}}_\xi
                          &=(N+\xi+1+\abs{n})\overline{a_\gamma^0}\|z^{\gamma+n}\|_\xi^2
                            =\ip{z^n}{Lz^{\gamma+n}}_\xi, \nonumber \\
				\ip{Lz^n}{z^{n-\gamma}}_\xi
				&=a_\gamma^0(\gamma\cdot n)\|z^{n-\gamma}\|_\xi^2
					=\ip{z^n}{Lz^{n-\gamma}}_\xi, \nonumber \\
				\ip{Lz^n}{z^{\gamma_1+n-\gamma_2}}_\xi&=a_{\gamma_2}^{\gamma_1}(\gamma_2\cdot n)\|z^{\gamma_1+n-\gamma_2}\|_\xi^2
					=\ip{z^n}Lz^{\gamma_1+n+\gamma_2}_\xi \nonumber
			\end{align}
			where $\abs{\gamma},\abs{\gamma_1},\abs{\gamma_2}=1$ and $\gamma_1\neq\gamma_2$.
			Also, $\ip{Lz^n}{z^m}=\ip{z^n}{Lz^m}=0$ for all other $m$.
			Since $n\in\mathbb{N}_0^N$ is arbitrary, this concludes that $L$ is symmetric.
		\end{proof}

                \subsubsection*{Proof of Theorem~\ref{mainthm} for the unit ball}

                Now, we approach finding the operators using the holomorphic
                discrete series
                representation of $\mathrm{SU}(N,1)$.	
		For any element $x\in\mathrm{SU}(N,1)$ denoted
			\begin{align}
				x:=
				\begin{bmatrix}
					a & b \\
					c^T & d
				\end{bmatrix} \nonumber
			\end{align}
			where $a\in M_N$, $b,c\in\mathbb{C}^{N}$ and $d\in\mathbb{C}$, consider the representation $\pi_\xi(x):\mathcal{A}_\xi^2(\mathbb{B}^N)\longrightarrow\mathcal{A}_\xi^2(\mathbb{B}^N)$ given by
			\begin{align}
				\pi_\xi(x)f(z):=\frac{1}{(-\ip{z}{b}+\overline{d})^{N+\xi+1}}f\left(\frac{a^*z-\overline{c}}{-\ip{z}{b}+\overline{d}}\right), \nonumber 
			\end{align}
			where $\ip{\cdot}{\cdot}$ here is the usual inner product on $\mathbb{C}^N$.
			It is known that $(\pi_\xi,\mathcal{A}_\xi^2(\mathbb{B}^N))$ is a unitary representation of $\mathrm{SU}(N,1)$ when $\xi>-1$ is an integer.
			Moreover, it defines a unitary representation of the universal covering group of $\mathrm{SU}(N,1)$ for all $\xi>-1$. In this paper we do not need to make a distinction between these groups, since their Lie algebras are the same and the exponential mapping is a local diffeomorphism.
		Let a general $X\in M_{N+1}$ be denoted with entries
		\begin{align}
			X:=
			\begin{bmatrix}
				x_{1,1} &  \cdots & x_{1,N+1} \\
				\vdots & \ddots & \vdots \\
				x_{N+1,1} & \cdots & x_{N+1,N+1}
			\end{bmatrix}. \nonumber
		\end{align}
		Then, $X\in\mathfrak{su}(N,1)$ if and only if
		\begin{align}
			&\begin{bmatrix}
				x_{1,1} & \cdots & x_{1,N} & -x_{1,N+1} \\
				\vdots & \ddots & \vdots & \vdots \\
				x_{N+1,1} & \cdots & x_{N+1,N} & -x_{N+1,N+1}
			\end{bmatrix} 
			=
			\begin{bmatrix}
				-\overline{x_{1,1}} & \cdots & -\overline{x_{N+1,1}} \\
				\vdots & \ddots & \vdots \\
				-\overline{x_{1,N}} & \cdots & -\overline{x_{N+1,N}} \\
				\overline{x_{1,N+1}} & \cdots & \overline{x_{N+1,N+1}}
			\end{bmatrix} \nonumber
		\end{align}
		and $\mathrm{tr}(X)=0$, i.e.,
		\begin{align}
			X&=
			\begin{bmatrix}
				x_{1,1} & x_{1,2} & x_{1,3} & \cdots & x_{1,N-1} & x_{1,N} & x_{1,N+1} \\
				-\overline{x_{1,2}} & x_{2,2} & x_{2,3} & \cdots & x_{2,N-1} & x_{2,N} & x_{2,N+1} \\
				-\overline{x_{1,3}} & -\overline{x_{2,3}} & x_{3,3} & \cdots & x_{3,N-1} & x_{3,N} & x_{3,N+1} \\ 
				\vdots & \vdots & \vdots & \ddots & \vdots & \vdots & \vdots \\
				-\overline{x_{1,N-1}} & -\overline{x_{2,N-1}} & -\overline{x_{3,N-1}} & \cdots & x_{N-1,N-1} & x_{N-1,N} & x_{N-1,N+1} \\
				-\overline{x_{1,N}} & -\overline{x_{2,N}} & -\overline{x_{3,N}} & \cdots & -\overline{x_{N-1,N}} & x_{N,N} & x_{N,N+1} \\[0.25em]
				\overline{x_{1,N+1}} & \overline{x_{2,N+1}} & \overline{x_{3,N+1}} & \cdots & \overline{x_{N-1,N+1}} & \overline{x_{N,N+1}} & -\sum_{j=1}^Nx_{j,j}\,\,
			\end{bmatrix} \nonumber 
		\end{align}
		where $x_{j,j}\in i\mathbb{R}$ for $j\in\{1,\cdots,N\}$.
		Equivalently, $X\in\mathfrak{su}(N,1)$ if and only if
		\begin{align}
			\begin{cases}
				x_{j,j}\in i\mathbb{R} & \text{if }j\in\{1,\cdots,N\}, \\
				x_{j,k}=-\overline{x_{k,j}}\in\mathbb{C} & \text{if }j>k \text{ and } j\in\{2,\cdots,N\}, \\
				x_{j,k}=\overline{x_{k,j}}\in\mathbb{C} & \text{if }j=N+1 \text{ and } k\in\{1,\cdots N\}.
			\end{cases} \nonumber
		\end{align} 
		Let $E_{j,k}$ denote the $(N+1)\times(N+1)$ matrix with $1$ in the $(j,k)$ entry and 0's elsewhere.
		A basis for $\mathfrak{su}(N,1)$ is the collection of matrices
		\begin{align}
			&\{iE_{j,j}-iE_{N+1,N+1}\,|\,1\leq j\leq N\}, \nonumber \\
			&\{E_{j,k}-E_{k,j}\,|\,j<k,2\leq k\leq N\}, \nonumber \\
			&\{iE_{j,k}+iE_{k,j}\,|\,j<k,2\leq k\leq N\}, \nonumber \\
			&\{E_{j,N+1}+E_{N+1,j}\,|\,1\leq j\leq N\}, \nonumber \\
			&\{iE_{j,N+1}-iE_{N+1,j}\,|\,1\leq j\leq N\}. \nonumber
		\end{align}
		These basis elements give us the operators
		\begin{align}
			\pi_\xi(\underbrace{iE_{j,j}-iE_{N+1,N+1}}_{\mathfrak{X}_1^j})
			&=-i(N+\xi+1)-2iz_j\frac{\partial}{\partial z_j}-\sum_{\substack{\ell=1 \\ \ell\neq j}}^Niz_\ell\frac{\partial}{\partial z_\ell}, \nonumber \\
			\pi_\xi(\underbrace{E_{j,k}-E_{k,j}}_{\mathfrak{X}_2^{j,k}})
			&=-z_k\frac{\partial}{\partial z_j}+z_j\frac{\partial}{\partial z_k}, \nonumber \\
			\pi_\xi(\underbrace{iE_{j,k}+iE_{k,j}}_{\mathfrak{X}_3^{j,k}})
			&=-iz_k\frac{\partial}{\partial z_j}-iz_j\frac{\partial}{\partial z_k}, \nonumber \\
			\pi_\xi(\underbrace{E_{j,N+1}+E_{N+1,j}}_{\mathfrak{X}_4^j})
			&=(N+\xi+1)z_j+z_j^2\frac{\partial}{\partial z_j}-\frac{\partial}{\partial z_j}+\sum_{\substack{\ell=1 \\ \ell\neq j}}^Nz_\ell z_j\frac{\partial}{\partial z_\ell}, \nonumber \\
			\pi_\xi(\underbrace{iE_{j,N+1}-iE_{N+1,j}}_{\mathfrak{X}_5^j})
			&=i(N+\xi+1)z_j+iz_j^2\frac{\partial}{\partial z_j}+i\frac{\partial}{\partial z_j}+\sum_{\substack{\ell=1 \\ \ell\neq j}}^Niz_\ell z_j\frac{\partial}{\partial z_\ell}. \nonumber
		\end{align}
		Notice that $\mathfrak{Y}\in\mathfrak{su}(N,1)$ if and only if there are $a_1^j,a_2^{j,k},a_3^{j,k},a_4^j,a_5^j\in\mathbb{R}$, such that 
			\begin{align}
				\mathfrak{Y}
				&=\sum_{j=1}^N[a_1^j\mathfrak{X}_1^j+a_4^j\mathfrak{X}_4^j+a_5^j\mathfrak{X}_5^j]+\sum_{k=2}^N\sum_{j=1}^{k-1}[a_2^{j,k}\mathfrak{X}_2^{j,k}+a_3^{j,k}\mathfrak{X}_3^{j,k}]. \nonumber 
			\end{align}
			Thus, any self-adjoint operator coming from the representation of $\mathfrak{su}(N,1)$ is given by $i\pi_\xi(\mathfrak{Y})$, which has form
			\begin{align}
				i\left(\sum_{j=1}^N[a_1^j\pi_\xi(\mathfrak{X}_1^j)+a_4^j\pi_\xi(\mathfrak{X}_4^j)+a_5^j\pi_\xi(\mathfrak{X}_5^j)]+\sum_{k=2}^N\sum_{j=1}^{k-1}[a_2^{j,k}\pi_\xi(\mathfrak{X}_2^{j,k})+a_3^{j,k}\pi_\xi(\mathfrak{X}_3^{j,k})]\right)  \nonumber
			\end{align}
			and can be expanded to 
			\begin{align}
				&(N+\xi+1)\sum_{j=1}^N\left[a_1^j+(ia_4^j-a_5^j)z_j\right] \nonumber \\
				&+\sum_{j=1}^N\left[(-ia_4^j-a_5^j)+\left(2a_1^j+\sum_{\substack{\ell=1 \\ \ell\neq j}}^Na_1^j\right)z_j+(ia_4^j-a_5^j)z_j^2+\sum_{\substack{\ell=1 \\ \ell\neq j}}^N(ia_4^\ell-a_5^\ell)z_\ell z_j\right]\frac{\partial}{\partial z_j} \nonumber \\
				&+\underbrace{\sum_{k=2}^N\sum_{j=1}^{k-1}\left[(-ia_2^{j,k}+a_3^{j,k})z_k\frac{\partial}{\partial z_j}+(ia_2^{j,k}+a_3^{j,k})z_j\frac{\partial}{\partial z_k}\right]}_{(*)} \nonumber
			\end{align}
			where $(*)$ can be written as
			\begin{align}
				\sum_{\omega=1}^N\left[\sum_{\ell=1}^{\omega-1}(ia_2^{\ell,\omega}+a_3^{\ell,\omega})z_\ell+\sum_{\ell=\omega+1}^N(-ia_2^{\omega,\ell}+a_3^{\omega,\ell})z_\ell\right]\frac{\partial}{\partial z_\omega}. \nonumber
			\end{align}
			Thus, we have the equivalent expression for $i\pi_\xi(\mathfrak{Y})$:
			\begin{align}
				i\pi_\xi(\mathfrak{Y})&=\left(\frac{N+\xi+1}{N+1}\sum_{\abs{\alpha}=1}a_\alpha^\alpha+(N+\xi+1)\sum_{\abs{\beta}=1}\overline{a_\beta^0}z^\beta\right) \nonumber \\
				&+\sum_{\abs{\alpha}=1}\left(a_\alpha^0+\sum_{\substack{\abs{\beta}=1}}\left(a_\alpha^\beta z^\beta+\overline{a_\beta^0}z^{\beta+\alpha}\right)\right)\left(\alpha\cdot\left(\frac{\partial}{\partial z_1},\cdots,\frac{\partial}{\partial z_N}\right)\right) \nonumber
			\end{align}
			where $a_\alpha^\alpha\in\mathbb{R}$.
			Notice that any operator in Theorem \ref{FODO} can be attained via translating a $i\pi_\xi(\mathfrak{Y})$ by some scaled identity, i.e., 
			\begin{align}
				L
				&=i\pi_\xi(\mathfrak{Y})-\left(\frac{N+\xi+1}{N+1}\sum_{\abs{\alpha}=1}a_\alpha^\alpha-a_0^0\right) \nonumber
			\end{align}
			where $a_0^0\in\mathbb{R}$, which shows that $(\pi_\xi,\mathcal{A}_\xi^2(\mathbb{B}^N))$ generates all the first-order self-adjoint differential operators on $\mathcal{A}_\xi^2(\mathbb{B}^N)$.

		\begin{remark} \label{ZhuRemark}
			Kehe Zhu in \cite{Zhu2015} asked if there exists self-adjoint $A,B$ on 
			$\mathcal{A}_\xi^2(\mathbb{D})$ ($\mathbb{D}$ the unit disk)
			such that $[A,B]=\lambda I$ for $\lambda\in\mathbb{C}\setminus\{0\}$, where
			$I$ is the identity operator.
			We showed in \cite{Deng2022} that this is not possible if we required $A,B$ be 
first-order self-adjoint differential operators.
			The answer is also negative when extended to $\mathcal{A}_\xi^2(\mathbb{B}^N)$, as that would require $a_\alpha^\alpha$ to be 0 for 
			every $\abs{\alpha}=1$, which means 
			$\frac{N+\xi+1}{N+1}\sum_{\omega=1}^Na_\alpha^\alpha=0$ as well.
		\end{remark}

                        \subsubsection*{Proof of Theorem~\ref{euler-isometry} for the unit ball}
                       
					To avoid introducing too many variables, we will re-use functions $f,g$ for different parts of this proof.
					It should be clear to the reader when the functions are being re-defined.
					Let $f\in\mathcal{A}_\xi^2(\mathbb{B}^N)$ and let $f(z):=\sum_{\beta\in\mathbb{N}_0^N}a^\beta z^\beta$.
					Notice that $\frac{(c+\abs{\beta})^2}{(N+\xi+2+\abs{\beta})(N+\xi+1+\abs{\beta})}$ converges
					to a positive number as $\abs{\beta}\to\infty$,
					so there is $C>0$ such that 
					\begin{align}
						\frac{1}{C}
						&\leq\frac{(c+\abs{\beta})^2}{(N+\xi+2+\abs{\beta})(N+\xi+1+\abs{\beta})}
							\leq C. \nonumber
					\end{align}
					Then, it follows that
					\begin{align}
						\norm{\widetilde{L}f}_{\xi+2}
						&=\sum_{\beta\in\mathbb{N}_0^N}(c+\abs{\beta})^2\abs{a^\beta}^2\norm{z^\beta}_{\xi+2}^2 \nonumber \\
						&=\sum_{\beta\in\mathbb{N}_0^N}(c+\abs{\beta})^2\frac{(N+\xi+2)(N+\xi+1)}{(N+\xi+2+\abs{\beta})(N+\xi+1+\abs{\beta})}\abs{a^\beta}^2\norm{z^\beta}_\xi^2 \nonumber \\
						&\leq(N+\xi+2)(N+\xi+1)C\sum_{\beta\in\mathbb{N}_0^N}\abs{a^\beta}^2\norm{z^\beta}_\xi^2 \nonumber \\
						&=(N+\xi+2)(N+\xi+1)C\norm{f}_\xi^2<\infty, \nonumber
					\end{align}
					so $\widetilde{L} f\in\mathcal{A}_{\xi+2}^2(\mathbb{B}^N)$, hence $\widetilde{L}$ is well-defined.
					Let $f,g\in\mathcal{A}_\xi^2(\mathbb{B}^N)$, where $f(z):=\sum_{\beta\in\mathbb{N}_0^N}a^\beta z^\beta$ and
					$g(z):=\sum_{\beta\in\mathbb{N}_0^N}b^\beta z^\beta$.	
					Clearly, 
					\begin{align}
						\widetilde{L}f(z)=\widetilde{L}g(z)
						&\iff\sum_{\beta\in\mathbb{N}_0^N}(c+\abs{\beta})a^\beta z^\beta=\sum_{\beta\in\mathbb{N}_0^N}(c+\abs{\beta})b^\beta z^\beta
							\iff a^\beta=b^\beta \nonumber 
					\end{align}
					for all $\beta\in\mathbb{N}_0^N$, so $\mathcal{L}$ is injective.
					Now, let $f\in\mathcal{A}_{\xi+2}^2(\mathbb{B}^N)$ and denote $f(z):=\sum_{\beta\in\mathbb{N}_0^N}a^\beta z^\beta$.
					Let us define $g(z):=\sum_{\beta\in\mathbb{N}_0^N}\frac{a^\beta}{c+\abs{\beta}}z^\beta$, which is well-defined because $c$
					is not 0 nor a negative integer, hence $c+\abs{\beta}\neq0$ for all $\beta\in\mathbb{N}_0^N$.
					Since
					\begin{align}
						\norm{g}_\xi^2
						&=\sum_{\beta\in\mathbb{N}_0^N}\frac{\abs{a^\beta}^2}{(c+\abs{\beta})^2}\norm{z^\beta}_\xi^2 \nonumber \\
						&\leq\frac{C}{(N+\xi+2)(N+\xi+1)}\sum_{\beta\in\mathbb{N}_0^N}\frac{(N+\xi+2)(N+\xi+1)\abs{a^\beta}^2}{(N+\xi+2+\abs{\beta})(N+\xi+1+\abs{\beta})}\norm{z^\beta}_\xi^2 \nonumber \\
						&=\frac{C}{(N+\xi+2)(N+\xi+1)}\sum_{\beta\in\mathbb{N}_0^N}\abs{a^\beta}^2\norm{z^\beta}_{\xi+2}^2 \nonumber \\
						&=\frac{C}{(N+\xi+2)(N+\xi+1)}\norm{f}_{\xi+2}^2
							<\infty, \nonumber
					\end{align}
					so $g\in\mathcal{A}_\xi^2(\mathbb{B}^N)$.
					Since $\widetilde{L}g(z)=f(z)$, we have that $\widetilde{L}$ is surjective.
					Finally, let $f\in\mathcal{A}_\xi^2(\mathbb{B}^N)$ and let $f(z)=\sum_{\beta\in\mathbb{N}_0^N}a^\beta z^\beta$.
					Then, 
					\begin{align}
						&\frac{(N+\xi+2)(N+\xi+1)}{C}\underbrace{\sum_{\beta\in\mathbb{N}_0^N}\abs{a^\beta}^2\norm{z^\beta}_\xi^2}_{=\norm{f}_\xi^2} \nonumber \\
						&\leq\underbrace{(N+\xi+2)(N+\xi+1)\sum_{\beta\in\mathbb{N}_0^N}\frac{(c+\abs{\beta})^2}{(N+\xi+2+\abs{\beta})(N+\xi+1+\abs{\beta})}\abs{a^\beta}^2\norm{z^\beta}_\xi^2}_{=\norm{\widetilde{L}f}_{\xi+2}^2} \nonumber \\
						&\leq(N+\xi+2)(N+\xi+1)C\underbrace{\sum_{\beta\in\mathbb{N}_0^N}\abs{a^\beta}^2\norm{z^\beta}_\xi^2}_{=\norm{f}_\xi^2} \nonumber
					\end{align}
					thus $\widetilde{L}$ is continuous.
					The argument for $\widetilde{L}^{-1}$ is linear, and continuity follows from the same argument using the $C$ bounds.
					Thus $\widetilde{L}$ is a linear homeomorphims.

\subsection{The case of the Generalized Unit Disk}

In this section we consider the domain of complex $2\times 2$ matrices
$Z$ for which $I-Z^*Z$ is positive definite which is the
same as assuming $I-ZZ^*>0$.
First, we will describe how the integral over $\mathbb{D}^2$ can be written
as a iterated integral over two unit balls $\mathbb{B}^2$ in $\mathbb{C}^2$.
Let
$Z=\begin{bmatrix} z_1 & z_2 \\z_3 & z_4
   \end{bmatrix} = [V \mid  W]$
   so $I-ZZ^*$ is equivalent to $I-VV^*-WW^*>0$. 
   Since $VV^*\geq 0$ we
   get that $I-WW^*$ is positive definite.
   Therefore, $I-WW^*=T$ for some positive definite operator
   $T$. Then $T$ is diagonalizable and there is
   a unitary matrix $U$ and a diagonal matrix $D$ such
   that $U^*T U = D$. We will find $U$ and $D$ in terms
   of $z_1,z_3$.
   Notice

   \begin{equation*}
     I-VV^* = \begin{bmatrix}
                1-|z_1|^2 & -z_1\overline{z}_3 \\
                -\overline{z}_1z_3 & 1-|z_3|^2
              \end{bmatrix}
            \end{equation*}
            which has eigenvalues
            $\lambda=1-|z_1|^2-|z_3|^2$ and $1$.
            The orthonormal
            eigenvectors are 
            $v_1=\frac{1}{\sqrt{|z_1|^2+|z_3|^2}}\begin{bmatrix} z_1 \\ z_3
                     \end{bmatrix}
                     $
                     and
                     $v_2= \frac{1}{\sqrt{|z_1|^2+|z_3|^2}}
            \begin{bmatrix}
              -\overline{z}_3 \\ \overline{z}_1
            \end{bmatrix}$.
            Thus
            $$I-VV^* =
            \frac{1}{|z_1|^2+|z_3|^2}
            \begin{bmatrix}
                -\overline{z}_3 & z_1\\
               \overline{z}_1 & z_3
            \end{bmatrix}
            \begin{bmatrix}
              1&0\\
              0& 1-|z_1|^2-|z_3|^2
            \end{bmatrix}
            \begin{bmatrix}
              \overline{z_3} & \overline{z}_1   \\
               z_1 & -z_3
            \end{bmatrix}
            $$
            So we can let $U=\frac{1}{\sqrt{|z_1|^2+|z_3|^2}}
            \begin{bmatrix}
                -\overline{z}_3 &z_1\\
              \overline{z}_1 & z_3
            \end{bmatrix}$
            and
            \begin{equation*}
              D= \begin{bmatrix}
              1 &  0\\
              0 & \lambda
            \end{bmatrix}
          \end{equation*}
          Notice that $\sqrt{T}=U^*\sqrt{D}U$
          and if we replace $W$ by $\sqrt{T}W_1$, then
          $I-VV^* -WW^*= I-VV^*-\sqrt{T}W_1W_1^*\sqrt{T}
          = T-\sqrt{T}W_1W_1^*\sqrt{T} =
          \sqrt{T}(I-W_1W_1^*)\sqrt{T}$. Since $\sqrt{T}$ is self-adjoint and
          strictly positive, $I-ZZ^*$ is positive definite if and only
          if $I-W_1W_1^*$ is. Notice that $I-W_1W_1^*>0$ is equivalent
          to $1-W_1^*W_1>0$, so this is the same as requiring $W_1$
          do be in the unit ball $B$ of $\mathbb{C}^2$.
          Lastly note that
          $\det(I-Z^*Z) = \det(I-ZZ^*) = \det(\sqrt{T}(I-W_1W_1^*)\sqrt{T})
          =\det(T)(1-W_1^*W_1)$ and $\det(T)=1-V^*V$.
          Equip the domain with the Euclidean measure $\mathrm{d}Z$
          inherited from $\mathbb{C}^4$ (or $\mathbb{R}^8$).
          Then, for $\xi>-1$,
          \begin{align}
            \int_{\mathbb{D}^2} f(Z)h(Z)^\xi \,dZ
            &= \int_{\mathbb{D}^2} f(Z)(\det(I-ZZ^*))^\xi \,dZ \\
            &= \int_{\mathbb{B}^2}    \int_{\mathbb{B}^2} f([V\mid \sqrt{T} W_1 ]) \det(T) \det(T)^{\xi}(1-W_1^*W_1)^\xi \,\mathrm{d}V \,\mathrm{d}W_1 \nonumber \\
            &= \int_{\mathbb{B}^2}    (1-V^*V)^{\xi+1}
              \int_{\mathbb{B}^2} f([V\mid \sqrt{T} W_1]) (1-W_1^*W_1)^\xi
              \,\mathrm{d}W_1\,\mathrm{d}V  \nonumber
          \end{align}
          Since the ball $\mathbb{B}^2$ is invariant under multiplication by
          a unitary matrix, we get
          \begin{equation*}
            \int_{\mathbb{D}^2} f(Z)\abs{\det(T)}^\xi\,dZ
            = \int_{\mathbb{B}^2}    (1-V^*V)^{\xi+1}
            \int_{\mathbb{B}^2} f([V \mid U\sqrt{D} W_1])
            (1-W_1^*W_1)^\xi \,\mathrm{d}W_1\,\mathrm{d}V.
          \end{equation*}
          Denoting $Z^n=z_1^{n_1}z_2^{n_2}z_3^{n_3}z_4^{n_4}$, we have
	\begin{align}
		\ip{Z^n}{Z^m}_\xi
		=&\int_{\mathbb{B}^2}\frac{z_1^{n_1}\overline{z_1^{m_1}}z_3^{n_3}\overline{z_3^{m_3}}}{(\sqrt{\abs{z_1}^2+\abs{z_3}^2})^{n_2+n_4+m_2+m_4}}(1-\abs{z_1}^2-\abs{z_3}^2)^{\xi+1} \nonumber \\
		&\int_{\mathbb{B}^2}(\sqrt{\lambda}z_1z_4-z_2\overline{z_3})^{n_2}(\sqrt{\lambda}\overline{z_1}\overline{z_4}-\overline{z_2}z_3)^{m_2} \nonumber \\
		&\qqquad(\sqrt{\lambda}z_3z_4+\overline{z_1}z_2)^{n_4}(\sqrt{\lambda}\overline{z_3}\overline{z_4}+z_1\overline{z_2})^{m_4} (1-|z_2|^2-|z_4|^2)^\xi \,\mathrm{d}W\,\mathrm{d}V \nonumber \\
		=&\sum_{\alpha=0}^{n_2}\sum_{\beta=0}^{m_2}\sum_{\gamma=0}^{n_4}\sum_{\varphi=0}^{m_4}\binom{n_2}{\alpha}\binom{m_2}{\beta}\binom{n_4}{\gamma}\binom{m_4}{\varphi}(-1)^{n_2+m_2-(\alpha+\beta)} \nonumber \\
		&\int_{\mathbb{B}^2}\frac{z_1^{n_1+\alpha+m_4-\varphi}\overline{z_1^{m_1+\beta+n_4-\gamma}}z_3^{n_3+m_2-\beta+\gamma}\overline{z_3^{m_3+n_2-\alpha+\varphi}}}{(\sqrt{\abs{z_1}^2+\abs{z_3}^2})^{n_2+n_4+m_2+m_4}}\nonumber \\
		&\qqquad(1-\abs{z_1}^2-\abs{z_3}^2)^{\xi+1+\frac{1}{2}(\alpha+\beta+\gamma+\varphi)}\,\mathrm{d}V \nonumber \\
		&\int_{\mathbb{B}^2}
			z_2^{n_2+n_4-(\alpha+\gamma)}\overline{z_2^{m_2+m_4-(\beta+\varphi)}}
			z_4^{\alpha+\gamma}\overline{z_4^{\beta+\varphi}}(1-|z_2|^2-|z_4|^2)^\xi\,\mathrm{d}W. \nonumber
	\end{align}
	We know that the monomials on the unit ball are orthogonal, which means that for the above integral to be non-zero, the following relations must all hold:
	\begin{align}
		n_1+\alpha+m_4-\varphi&=m_1+\beta+n_4-\gamma, \nonumber \\
		n_2+n_4-(\alpha+\gamma)&=m_2+m_4-(\beta+\varphi), \nonumber \\
		n_3+m_2-\beta+\gamma&=m_3+n_2-\alpha+\varphi, \nonumber \\
		\alpha+\gamma&=\beta+\varphi. \nonumber
	\end{align}
	Combining the relations, we get that the inner-product is only non-zero when
	\begin{align}
		n_2+n_4&=m_2+m_4, \nonumber \\
		n_1+n_2&=m_1+m_2, \nonumber \\
		n_2+m_3&=n_3+m_2. \nonumber
	\end{align}
	Also, by combining the three relations above, we get that
	\begin{align}
		n_1+n_2+n_3+n_4=m_1+m_2+m_3+m_4, \nonumber
	\end{align}
	which also shows the already known statement
        that monomials of different homogenous degrees are orthogonal.
	
	\begin{proposition}\label{LforGenDisk}
          Let $f_\alpha\in\mathcal{A}_\xi^2(\mathbb{D}^2)$
          be denoted $f_\alpha(z)=\sum_{\beta\in\mathbb{N}_0^4}a_\alpha^\beta z^\beta$ where $\alpha\in\mathbb{N}_0^4$. 
		If the operator $L=f_0+\sum_{k=1}^4 f_{e_k} \frac{\partial}{\partial z_k}$ with domain $\mathcal{D}(L)=P(\mathbb{D}^2)$ is symmetric, then
		\begin{equation*}
                 (*)\,\, \begin{cases}
                	f_0\,=a_0^0+(\xi+4)\sum_{k=1}^4\overline{a_{e_k}^0}z^{e_k}, \\
			f_{e_1}=a_{e_1}^0+a_{e_1}^{e_1}z^{e_1}+a_{e_1}^{e_2}z^{e_2}+a_{e_1}^{e_3}z^{e_3}+\overline{a_{e_4}^0}z^{e_2+e_3}+\sum_{k\neq4} \overline{a_{e_k}^0}z^{e_1+e_k}, \\
			f_{e_2}=a_{e_2}^0+a_{e_2}^{e_2}z^{e_2}+\overline{a_{e_1}^{e_2}}z^{e_1}+a_{e_1}^{e_3}z^{e_4}+\overline{a_{e_3}^0}z^{e_1+e_4}+\sum_{k\neq 3} \overline{a_{e_k}^0}z^{e_2+e_k},  \\
			f_{e_3}=a_{e_3}^0+a_{e_3}^{e_3}z^{e_3}+a_{e_1}^{e_2}z^{e_4}+\overline{a_{e_1}^{e_3}}z^{e_1}+\overline{a_{e_2}^0}z^{e_1+e_4}+\sum_{k\neq2} \overline{a_{e_k}^0}z^{e_3+e_k}, \\
			f_{e_4}=a_{e_4}^0+a_{e_4}^{e_4}z^{e_4}+\overline{a_{e_1}^{e_2}}z^{e_3}+\overline{a_{e_1}^{e_3}}z^{e_2}+\overline{a_{e_1}^0}z^{e_2+e_3}+\sum_{k\neq1}\overline{a_{e_k}^0}z^{e_4+e_k} 		    \\
                 \text{ with }  a_0^0,a_{e_i}^{e_i}\in\mathbb{R},\quad  a_0^{e_i}=(\xi+4)\overline{a_{e_i}^0}, \quad a_{e_1}^{e_1}+a_{e_4}^{e_4}=a_{e_2}^{e_2}+a_{e_3}^{e_3}.        
                  \end{cases}
		\end{equation*}
	 
	\end{proposition} \label{ball main theorem}
	\begin{proof}
          Assume $L$ is symmetric, so given $n,m\in\mathbb{N}_0^N$, we have $\ip{Lz^n}{z^m}_\xi=\ip{z^n}{Lz^m}_\xi$.
          This equality is the same as
          $$
 \sum_{|\alpha|\leq 1} \sum_{\beta\geq 0}
a_\alpha^\beta c_{\alpha,\gamma} \langle z^{\beta+\gamma-\alpha},z^\rho\rangle_\xi
= \sum_{|\alpha|\leq 1} \sum_{\beta\geq 0}
\overline{a_\alpha^\beta} c_{\alpha,\rho} \langle z^\gamma,z^{\beta+\rho-\alpha}\rangle_\xi,
$$
where $$c_{\alpha,\gamma} =
\begin{cases}
  1 & \alpha=0 \\
  \alpha\cdot \gamma & \gamma-\alpha\geq 0 \\
  0 & \text{else}.
\end{cases}
$$
Plugging in different values of $n$ and $m$ it is possible to determine
the properties of the coefficients of $L$.
Below is a table summarizing the relations and which $m$ and $n$ were
used to derive them.
\renewcommand{\arraystretch}{1.5}

		\begin{table}[H]
			\center
			\begin{tabular}{|c|c|c|}
				\hline
				Number & Relation & Deriving the Relation \\
				\hline\hline
				(1) & $a_0^0\in\mathbb{R}$ & $n=m=0$ \\
				\hline
				(2) & $a_0^{e_i}=(\xi+4)\overline{a_{e_i}^0}$ & $n=e_i, m=0$ \\
				\hline
				(3) & $a_{e_i}^{e_i}\in\mathbb{R}$ & $n=m=e_i$ \\
				\hline
				(4) & $a_{e_i}^{e_{5-i}}=0$ & $n=2e_i, m=e_i+e_{5-i}$ and (5) \\
				\hline
				(5) & $a_{e_i}^{e_j}=\overline{a_{e_j}^{e_i}}$ & $n=e_i, m=e_j$ \\
				\hline
				(6) & $a_{e_i}^{e_j}=a_{e_{5-j}}^{e_{5-i}}$ & $n=e_i+e_{5-i},m=e_i+e_j$ \\
				\hline
				(7) & $a_{e_i}^{2e_i}=\overline{a_{e_i}^0}$ & $n=2e_i, m=e_i$ \\
				\hline
				(8) & $a_{e_i}^{2e_j}=0$ & $n=2e_i, m=e_j$ \\
				\hline
				(9) & $a_{e_i}^{e_i+e_{5-i}}=0$ & $n=e_i+e_{5-i}, m=e_j$ and (2) and (10)\\
				\hline
				(10) & $a_{e_i}^{e_j+e_{5-j}}=\overline{a_{e_{5-i}}^0}$ & $n=e_i+e_{5-i},m=e_i$ \\
				\hline
				(11) & $a_{e_i}^{e_i+e_j}=\overline{a_{e_j}^0}$ & $n=e_i+e_j,m=e_i$ and (2) \\
				\hline
				(12) & $a_{e_i}^{e_{5-i}+e_j}=0$ & $n=e_i+e_j, m=e_{5-i}$ \\
				\hline
				(13) & $a_{e_i}^{e_i}+a_{e_{5-i}}^{e_{5-i}}=a_{e_j}^{e_j}+a_{e_{5-j}}^{e_{5-j}}$ & $n=e_i+e_{5-i},m=e_j+e_{5-j}$ \\
				\hline
			\end{tabular}
		\end{table}
		where $1\leq i\leq 4$ and $j\neq i,5-i$. 
		Combined with Proposition \ref{prop1}, we have our desired $L$.

In the process we have used 
the following norms
\begin{align}
  \|z^0\|^2_\xi&=\frac{\pi^4}{2(\xi+2)(\xi+3)}, \nonumber \\
  \|z^{e_i}\|_\xi^2&=\frac{\pi^4(\xi+1)!}{2(\xi+4)!}, \nonumber \\
  \|z^{2e_i}\|_\xi^2&=\frac{\pi^4(\xi+1)!}{(\xi+5)!}, \nonumber \\
  \|z^{e_i+e_j}\|_\xi^2&=\frac{\pi^4(\xi+1)!}{2(\xi+5)!}, j\neq i,5-i \nonumber \\
  \|z^{e_i+e_{5-i}}\|_\xi^2&=\frac{\pi^4(\xi+1)!(\xi+4)}{6(\xi+5)!}, \nonumber \\
\intertext{and fact that the inner product for $\alpha<\beta$ with $|\alpha|=|\beta|=2$ satisfies}
  \ip{z^{\alpha}}{z^{\beta}}_\xi&=                            
  \begin{cases}
    -\frac{\pi^4(\xi+1)!}{6(\xi+5)!}         & \alpha=e_1+e_4,\beta=e_2+e_3, \\
    0 & \text{else}
  \end{cases}
  \nonumber
\end{align}
as desired.

\end{proof}

\subsubsection*{Proof of Theorem~\ref{mainthm} for the generalized disk}

              For $x\in\mathrm{SU}(2,2)$ and $A,B,C,D\in M_2$, denoted
	\begin{align}
		x^{-1}
		&:=
			\begin{bmatrix}
				A & B \\
				C & D
			\end{bmatrix}, \nonumber
	\end{align}
	consider the discrete series representation $\pi_\xi(x):\mathcal{A}_\xi^2(\mathbb{D}^2)\longrightarrow\mathcal{A}_\xi^2(\mathbb{D}^2)$ given by
	\begin{align}
		\pi_\xi(x)f(Z)
		&=\det(CZ+D)^{-(\xi+4)}f((AZ+B)(CZ+D)^{-1}). \nonumber
	\end{align}
	Similar to the case of $\mathrm{SU}(N,1)$, the representation $(\pi_\xi,\mathcal{A}_\xi^2(\mathbb{D}^2))$ is a unitary representation and it defines a unitary representation of the universal covering group of $\mathrm{SU}(2,2)$.
	Again, we do not need to make a distinction between these groups, since their Lie algebras are the same and the exponential mapping is a local diffeomorphism.
	
	We now approach finding the differential operators that come from the Lie Algebra representation of $\mathrm{SU}(2,2)$.
	Note that $W\in\mathfrak{su}(2,2)$ if and only if 
	\begin{align}
		\begin{bmatrix}
			w_{1,1} & w_{1,2} & -w_{1,3} & -w_{1,4} \\
			w_{2,1} & w_{2,2} & -w_{2,3} & -w_{2,4} \\
			w_{3,1} & w_{3,2} & -w_{3,3} & -w_{3,4} \\ 
			w_{4,1} & w_{4,2} & -w_{4,3} & -w_{4,4}
		\end{bmatrix}
		=
		\begin{bmatrix}
			-\overline{w_{1,1}} & -\overline{w_{2,1}} & -\overline{w_{3,1}} & -\overline{w_{4,1}} \\
			-\overline{w_{1,2}} & -\overline{w_{2,2}} & -\overline{w_{3,2}} & -\overline{w_{4,2}} \\
			\overline{w_{1,3}} & \overline{w_{2,3}} & \overline{w_{3,3}} & \overline{w_{4,3}} \\
			\overline{w_{1,4}} & \overline{w_{2,4}} & \overline{w_{3,4}} & \overline{w_{4,4}} 
		\end{bmatrix} \nonumber
	\end{align}
	and $\mathrm{tr}(W)=0$, i.e.,
	\begin{align}
		W=
		\begin{bmatrix}
			w_{1,1} & w_{1,2} & w_{1,3} & w_{1,4} \\
			-\overline{w_{1,2}} & w_{2,2} & w_{2,3} & w_{2,4} \\
			\overline{w_{1,3}} & \overline{w_{2,3}} & w_{3,3} & w_{3,4} \\
			\overline{w_{1,4}} & \overline{w_{2,4}} & -\overline{w_{3,4}} & -\sum_{j=1}^3w_{j,j}
		\end{bmatrix} \nonumber
	\end{align}
	where $w_{j,j}\in i\mathbb{R}$.
	Then, a basis for $\mathfrak{su}(2,2)$ is the collection of matrices
	\begin{align}
		\mathfrak{A}_1&=
		\begin{bmatrix}
			i & 0 & 0 & 0 \\
			0 & 0 & 0 & 0 \\
			0 & 0 & 0 & 0 \\
			0 & 0 & 0 & -i
		\end{bmatrix},
		\mathfrak{A}_2=
		\begin{bmatrix}
			0 & 0 & 0 & 0 \\
			0 & i & 0 & 0 \\
			0 & 0 & 0 & 0 \\
			0 & 0 & 0 & -i
		\end{bmatrix},
		\mathfrak{A}_3=
		\begin{bmatrix}
			0 & 0 & 0 & 0 \\
			0 & 0 & 0 & 0 \\
			0 & 0 & i & 0 \\
			0 & 0 & 0 & -i
		\end{bmatrix}, \nonumber \\
		\mathfrak{A}_4&=
		\begin{bmatrix}
			0 & 1 & 0 & 0 \\
			-1 & 0 & 0 & 0 \\
			0 & 0 & 0 & 0 \\
			0 & 0 & 0 & 0
		\end{bmatrix},
		\mathfrak{A}_5=
		\begin{bmatrix}
			0 & 0 & 0 & 0 \\
			0 & 0 & 0 & 0 \\
			0 & 0 & 0 & 1 \\
			0 & 0 & -1 & 0
		\end{bmatrix},
		\mathfrak{A}_6=
		\begin{bmatrix}
			0 & i & 0 & 0 \\
			i & 0 & 0 & 0 \\
			0 & 0 & 0 & 0 \\
			0 & 0 & 0 & 0
		\end{bmatrix}, \nonumber \\
		\mathfrak{A}_7&=
		\begin{bmatrix}
			0 & 0 & 0 & 0 \\
			0 & 0 & 0 & 0 \\
			0 & 0 & 0 & i \\
			0 & 0 & i & 0
		\end{bmatrix},
		\mathfrak{A}_8=
		\begin{bmatrix}
			0 & 0 & 1 & 0 \\
			0 & 0 & 0 & 0 \\
			1 & 0 & 0 & 0 \\
			0 & 0 & 0 & 0
		\end{bmatrix},
		\mathfrak{A}_9=
		\begin{bmatrix}
			0 & 0 & 0 & 1 \\
			0 & 0 & 0 & 0 \\
			0 & 0 & 0 & 0 \\
			1 & 0 & 0 & 0
		\end{bmatrix}, \nonumber \\
		\mathfrak{A}_{10}&=
		\begin{bmatrix}
			0 & 0 & 0 & 0 \\
			0 & 0 & 1 & 0 \\
			0 & 1 & 0 & 0 \\
			0 & 0 & 0 & 0
		\end{bmatrix},
		\mathfrak{A}_{11}=
		\begin{bmatrix}
			0 & 0 & 0 & 0 \\
			0 & 0 & 0 & 1 \\
			0 & 0 & 0 & 0 \\
			0 & 1 & 0 & 0
		\end{bmatrix},
		\mathfrak{A}_{12}=
		\begin{bmatrix}
			0 & 0 & i & 0 \\
			0 & 0 & 0 & 0 \\
			-i & 0 & 0 & 0 \\
			0 & 0 & 0 & 0
		\end{bmatrix}, \nonumber \\
		\mathfrak{A}_{13}&=
		\begin{bmatrix}
			0 & 0 & 0 & i \\
			0 & 0 & 0 & 0 \\
			0 & 0 & 0 & 0 \\
			-i & 0 & 0 & 0
		\end{bmatrix}, 
		\mathfrak{A}_{14}=
		\begin{bmatrix}
			0 & 0 & 0 & 0 \\
			0 & 0 & i & 0 \\
			0 & -i & 0 & 0 \\
			0 & 0 & 0 & 0
		\end{bmatrix},
		\mathfrak{A}_{15}=
		\begin{bmatrix}
			0 & 0 & 0 & 0 \\
			0 & 0 & 0 & i \\
			0 & 0 & 0 & 0 \\
			0 & -i & 0 & 0
		\end{bmatrix}, \nonumber
	\end{align}
	and the corresponding skew-adjoint operators from the Lie Algebra are
	\begin{align}
		\pi_\xi(\mathfrak{A}_1^{-1})&=(\xi+4)i+iz_1\frac{\partial f}{\partial z_1}+2iz_2\frac{\partial f}{\partial z_2}+iz_4\frac{\partial f}{\partial z_4}, \nonumber \\
		\pi_\xi(\mathfrak{A}_2^{-1})&=(\xi+4)i+iz_2\frac{\partial f}{\partial z_2}+iz_3\frac{\partial f}{\partial z_3}+2iz_4\frac{\partial f}{\partial z_4}, \nonumber \\
		\pi_\xi(\mathfrak{A}_3^{-1})&=-iz_1\frac{\partial f}{\partial z_1}+iz_2\frac{\partial f}{\partial z_2}-iz_3\frac{\partial f}{\partial z_3}+iz_4\frac{\partial f}{\partial z_4}, \nonumber \\
		\pi_\xi(\mathfrak{A}_4^{-1})&=z_3\frac{\partial f}{\partial z_1}+z_4\frac{\partial f}{\partial z_2}-z_1\frac{\partial f}{\partial z_3}-z_2\frac{\partial f}{\partial z_4}, \nonumber \\
		\pi_\xi(\mathfrak{A}_5^{-1})&=z_2\frac{\partial f}{\partial z_1}-z_1\frac{\partial f}{\partial z_2}+z_4\frac{\partial f}{\partial z_3}-z_3\frac{\partial f}{\partial z_4}, \nonumber \\
		\pi_\xi(\mathfrak{A}_6^{-1})&=iz_3\frac{\partial f}{\partial z_1}+iz_4\frac{\partial f}{\partial z_2}+iz_1\frac{\partial f}{\partial z_3}+iz_2\frac{\partial f}{\partial z_4}, \nonumber \\
		\pi_\xi(\mathfrak{A}_7^{-1})&=-iz_2\frac{\partial f}{\partial z_1}-iz_1\frac{\partial f}{\partial z_2}-iz_4\frac{\partial f}{\partial z_3}-iz_3\frac{\partial f}{\partial z_4}, \nonumber \\
		\pi_\xi(\mathfrak{A}_8^{-1})&=-(\xi+4)z_1+(1-z_1^2)\frac{\partial f}{\partial z_1}-z_1z_2\frac{\partial f}{\partial z_2}-z_1z_3\frac{\partial f}{\partial z_3}-z_2z_3\frac{\partial f}{\partial z_4}, \nonumber \\
		\pi_\xi(\mathfrak{A}_9^{-1})&=-(\xi+4)z_2-z_1z_2\frac{\partial f}{\partial z_1}+(1-z_2^2)\frac{\partial f}{\partial z_2}-z_1z_4\frac{\partial f}{\partial z_3}-z_2z_4\frac{\partial f}{\partial z_4}, \nonumber \\
		\pi_\xi(\mathfrak{A}_{10}^{-1})&=-(\xi+4)z_3-z_1z_3\frac{\partial f}{\partial z_1}-z_1z_4\frac{\partial f}{\partial z_2}+(1-z_3^2)\frac{\partial f}{\partial z_3}-z_3z_4\frac{\partial f}{\partial z_4}, \nonumber \\
		\pi_\xi(\mathfrak{A}_{11}^{-1})&=-(\xi+4)z_4-z_2z_3\frac{\partial f}{\partial z_1}-z_2z_4\frac{\partial f}{\partial z_2}-z_3z_4\frac{\partial f}{\partial z_3}+(1-z_4^2)\frac{\partial f}{\partial z_4}, \nonumber \\
		\pi_\xi(\mathfrak{A}_{12}^{-1})&=(\xi+4)iz_1+i(1+z_1^2)\frac{\partial f}{\partial z_1}+iz_1z_2\frac{\partial f}{\partial z_2}+iz_1z_3\frac{\partial f}{\partial z_3}+iz_2z_3\frac{\partial f}{\partial z_4}, \nonumber \\
		\pi_\xi(\mathfrak{A}_{13}^{-1})&=(\xi+4)iz_2+iz_1z_2\frac{\partial f}{\partial z_1}+i(1+z_2^2)\frac{\partial f}{\partial z_2}+iz_1z_4\frac{\partial f}{\partial z_3}+iz_2z_4\frac{\partial f}{\partial z_4}, \nonumber \\
		\pi_\xi(\mathfrak{A}_{14}^{-1})&=(\xi+4)iz_3+iz_1z_3\frac{\partial f}{\partial z_1}+iz_1z_4\frac{\partial f}{\partial z_2}+i(1+z_3^2)\frac{\partial f}{\partial z_3}+iz_3z_4\frac{\partial f}{\partial z_4}, \nonumber \\
		\pi_\xi(\mathfrak{A}_{15}^{-1})&=(\xi+4)iz_4+iz_2z_3\frac{\partial f}{\partial z_1}+iz_2z_4\frac{\partial f}{\partial z_2}+iz_3z_4\frac{\partial f}{\partial z_3}+i(1+z_4^2)\frac{\partial f}{\partial z_4}. \nonumber
	\end{align}
	Since $\mathfrak{Y}\in\mathfrak{su}(2,2)$ if and only if there are $a_1,\cdots,a_{15}\in\mathbb{R}$ such that 
	\begin{align}
		\mathfrak{Y}=\sum_{k=1}^{15}a_k\pi_\xi(\mathfrak{A}_k), \nonumber
	\end{align}
	then any self-adjoint operator coming form the representation of $\mathfrak{su}(2,2)$ has form 
	\begin{align}
		i\pi_\xi(\mathfrak{Y})
		&=i\sum_{k=1}^{15}a_k\pi_\xi(\mathfrak{A}_k). \nonumber 
	\end{align}

        If $L$ is an operator satisfying $(*)$ from
        Proposition~\ref{LforGenDisk}, then we can
        choose $a_i$'s to satisfy the following system of equations:
	\begin{align}
		a_{e_1}^0&=-a_{12}+ia_8, 
		&&a_{e_2}^0=-a_{13}+ia_9, \nonumber \\
		a_{e_3}^0&=-a_{14}+ia_{10}, 
		&&a_{e_4}^0=-a_{15}+ia_{11}, \nonumber \\
		a_{e_1}^{e_1}&=a_3-a_1, 
		&&a_{e_2}^{e_2}=-2a_1-a_2-a_3, \nonumber \\
		a_{e_3}^{e_3}&=a_3-a_2, 
		&&a_{e_4}^{e_4}=-a_1-2a_2-a_3, \nonumber \\
		a_{e_1}^{e_2}&=a_7+ia_5, 
		&&a_{e_1}^{e_3}=-a_6+ia_4, \nonumber
	\end{align}
	i.e., we choose the $a_i$'s in the following way:
	\begin{align}
		a_1&=-\frac{a_{e_1}^{e_1}}{2}-\frac{a_{e_2}^{e_2}-a_{e_3}^{e_3}}{4}, 
		&&a_2=\frac{a_{e_1}^{e_1}}{2}-\frac{a_{e_2}^{e_2}}{4}-\frac{3a_{e_3}^{e_3}}{4}, 
		&&a_3=\frac{2a_{e_1}^{e_1}-a_{e_2}^{e_2}+a_{e_3}^{e_3}}{4}, \nonumber \\
		a_4&=\mathrm{Im}(a_{e_1}^{e_3}),
		&&a_5=\mathrm{Im}(a_{e_1}^{e_2}), 
		&&a_6=-\mathrm{Re}(a_{e_1}^{e_3}), \nonumber \\
		a_7&=\mathrm{Re}(a_{e_1}^{e_2}), 
		&&a_8=\mathrm{Im}(a_{e_1}^0), 
		&&a_9=\mathrm{Im}(a_{e_2}^0), \nonumber \\
		a_{10}&=\mathrm{Im}(a_{e_3}^0), 
		&&a_{11}=\mathrm{Im}(a_{e_4}^0), 
		&&a_{12}=-\mathrm{Re}(a_{e_1}^0), \nonumber \\
		a_{13}&=-\mathrm{Re}(a_{e_2}^0), 
		&&a_{14}=-\mathrm{Re}(a_{e_3}^0), 
		&&a_{15}=-\mathrm{Re}(a_{e_4}^0). \nonumber 
	\end{align}
	Then $L$ can be written as
	\begin{align*}
		L=i\pi_\xi(\mathfrak{Y})+(3a_1+3a_2)+a_0^0
	\end{align*}
        when the right hand side is restricted to the polynomials.
	
        Since the polynomials are a core for $\pi_\xi(\mathfrak{Y})$
        it follows that if $L$ satisfies
        $(*)$ of Proposition~\ref{LforGenDisk}
        then $L$ is symmetric.
        
	\begin{remark}
          Notice that this argument uses representation theory
          to overcome the difficulty of showing that $L$ satisfying
          $(*)$ of Proposition~\ref{LforGenDisk} is symmetric.
          This is especially useful as we avoid calculating
          inner products of monomials above degree 2 (recall that the 
		  inner products of monomials with
		  same homogenous degree may not be orthogonal on $\mathbb{D}^2$).
          In some sense this means that we just need $L$ to be symmetric
          with respect to polynomials of low degree (which is not a dense set).
	\end{remark}

\appendix
\section{Relations} \label{relations}
	For $n=m=0$, we have
	\begin{align}
		\ip{Lz^0}{z^0}_\xi
		&=\sum_{\beta\in\mathbb{N}_0^4}a_0^\beta\ip{z^\beta}{z^0}_\xi
			=a_0^0\norm{z^0}^2_\xi
			=a_0^0\frac{\pi^4}{2(\xi+2)(\xi+3)}, \nonumber \\
		\ip{z^0}{Lz^0}_\xi
		&=\sum_{\beta\in\mathbb{N}_0^4}\overline{a_0^\beta}\overline{\ip{z^\beta}{z^0}_\xi}
			=\overline{a_0^0}\norm{z^0}^2_\xi
			=\overline{a_0^0}\frac{\pi^4}{2(\xi+2)(\xi+3)}, \nonumber
	\end{align}
	so $a_0^0=\overline{a_0^0}\in\mathbb{R}$.
	For $n= e_i$ and $m=0$, we have
	\begin{align}
		\ip{Lz^{ e_i}}{z^0}_\xi
		&=\sum_{\beta\in\mathbb{N}_0^4}a_0^\beta\ip{z^{\beta+ e_i}}{z^0}_\xi+\sum_{\abs{\alpha}=1}\sum_{\beta\in\mathbb{N}_0^4}a_\alpha^\beta(\alpha\cdot\gamma)\ip{z^{\beta+\gamma-\alpha}}{z^0}_\xi \nonumber \\
		&=a_{ e_i}^0\norm{z^0}^2_\xi
			=a_{ e_i}^0\frac{\pi^4}{2(\xi+2)(\xi+3)}, \nonumber \\
		\ip{z^{ e_i}}{Lz^0}_\xi
		&=\sum_{\beta\in\mathbb{N}_0^4}\overline{a_0^\beta}\overline{\ip{z^{\beta}}{z^{ e_i}}_\xi}
			=\overline{a_0^{ e_i}}\norm{z^{ e_i}}^2_\xi
			=\overline{a_0^{ e_i}}\frac{\pi^4(\xi+1)!}{2(\xi+4)!}, \nonumber
	\end{align}
	so $a_0^{ e_i}=(\xi+4)\overline{a_{ e_i}^0}$.
	For $n= e_i$ and $m= e_i$, we have
	\begin{align}
		\ip{Lz^{ e_i}}{z^{ e_i}}_\xi
		&=\sum_{\beta\in\mathbb{N}_0^4}a_0^\beta\ip{z^{\beta+ e_i}}{z^{ e_i}}_\xi+\sum_{\abs{\alpha}=1}\sum_{\beta\in\mathbb{N}_0^4}a_\alpha^\beta(\alpha\cdot e_i)\ip{z^{\beta+ e_i-\alpha}}{z^{ e_i}}_\xi \nonumber \\
		&=a_0^0\norm{z^{ e_i}}^2_\xi+a_{ e_i}^{ e_i}\norm{z^{ e_i}}^2_\xi
			=a_0^0\frac{\pi^4(\xi+1)!}{2(\xi+4)!}+a_{ e_i}^{ e_i}\frac{\pi^4(\xi+1)!}{2(\xi+4)!}, \nonumber \\
		\ip{z^{ e_i}}{Lz^{ e_i}}_\xi
		&=\sum_{\beta\in\mathbb{N}_0^4}\overline{a_0^\beta}\overline{\ip{z^{\beta+ e_i}}{z^{ e_i}}_\xi}+\sum_{\abs{\alpha}=1}\sum_{\beta\in\mathbb{N}_0^4}\overline{a_\alpha^\beta}(\alpha\cdot e_i)\overline{\ip{z^{\beta+ e_i-\alpha}}{z^{ e_i}}_\xi} \nonumber \\
		&=\overline{a_0^0}\norm{z^{ e_i}}^2_\xi+\overline{a_{ e_i}^{ e_i}}\norm{z^{ e_i}}^2_\xi
			=\overline{a_0^0}\frac{\pi^4(\xi+1)!}{2(\xi+4)!}+\overline{a_{ e_i}^{ e_i}}\frac{\pi^4(\xi+1)!}{2(\xi+4)!}, \nonumber
	\end{align}
	so it follows that $a_{ e_i}^{ e_i}=\overline{a_{ e_i}^{ e_i}}\in\mathbb{R}$. 
	For $n= e_i$ and $m= e_j$, we have
	\begin{align}
		\ip{Lz^{ e_i}}{z^{ e_j}}_\xi
		&=\sum_{\beta\in\mathbb{N}_0^4}a_0^\beta\ip{z^{\beta+ e_i}}{z^{ e_j}}_\xi+\sum_{\abs{\alpha}=1}\sum_{\beta\in\mathbb{N}_0^4}a_\alpha^\beta(\alpha\cdot e_i)\ip{z^{\beta+ e_i-\alpha}}{z^{ e_j}}_\xi \nonumber \\
		&=a_{ e_i}^{ e_j}\norm{z^{ e_j}}^2_\xi
			=a_{ e_i}^{ e_j}\frac{\pi^4(\xi+1)!}{2(\xi+4)!}, \nonumber \\
		\ip{z^{ e_i}}{Lz^{ e_j}}_\xi
		&=\sum_{\beta\in\mathbb{N}_0^4}\overline{a_0^\beta}\overline{\ip{z^{\beta+ e_j}}{z^{ e_i}}_\xi}+\sum_{\abs{\alpha}=1}\sum_{\beta\in\mathbb{N}_0^4}\overline{a_\alpha^\beta}(\alpha\cdot e_j)\overline{\ip{z^{\beta+ e_j-\alpha}}{z^{ e_i}}_\xi} \nonumber \\
		&=\overline{a_{ e_j}^{ e_i}}\norm{z^{ e_i}}^2_\xi
			=\overline{a_{ e_j}^{ e_i}}\frac{\pi^4(\xi+1)!}{2(\xi+4)!}, \nonumber
	\end{align}
	which gives the relation $a_{ e_i}^{ e_j}=\overline{a_{ e_j}^{ e_i}}$.
	For $n=2 e_i$ and $m= e_i$, we have
	\begin{align}
		\ip{Lz^{2 e_i}}{z^{ e_i}}_\xi
		&=\sum_{\beta\in\mathbb{N}_0^4}a_0^\beta\ip{z^{\beta+2 e_i}}{z^{ e_i}}_\xi+\sum_{\abs{\alpha}=1}\sum_{\beta\in\mathbb{N}_0^4}a_\alpha^\beta(\alpha\cdot(2 e_i))\ip{z^{\beta+2 e_i-\alpha}}{z^{ e_i}}_\xi \nonumber \\
		&=2a_{ e_i}^0\norm{z^{ e_i}}^2_\xi
			=a_{ e_i}^0\frac{\pi^4(\xi+1)!}{(\xi+4)!}, \nonumber \\
		\ip{z^{2 e_i}}{Lz^{ e_i}}_\xi
		&=\sum_{\beta\in\mathbb{N}_0^4}\overline{a_0^\beta}\overline{\ip{z^{\beta+ e_i}}{z^{2 e_i}}_\xi}+\sum_{\abs{\alpha}=1}\sum_{\beta\in\mathbb{N}_0^4}\overline{a_\alpha^\beta}(\alpha\cdot e_i)\overline{\ip{z^{\beta+ e_i-\alpha}}{z^{2 e_i}}_\xi} \nonumber \\
		&=\overline{a_0^{ e_i}}\norm{z^{2 e_i}}^2_\xi+\overline{a^{2 e_i}_{ e_i}}\norm{z^{2 e_i}}^2_\xi
			=\overline{a_0^{ e_i}}\frac{\pi^4(\xi+1)!}{(\xi+5)!}+\overline{a_{ e_i}^{2 e_i}}\frac{\pi^4(\xi+1)!}{(\xi+5)!}, \nonumber
	\end{align}
	so we have that $a_{ e_i}^{2 e_i}=(\xi+5)\overline{a_{ e_i}^0}-a^{ e_i}_0=\overline{a^{ e_i}_0}$.
	For $n=2 e_i$ and $m= e_j$, we have
	\begin{align}
		\ip{Lz^{2 e_i}}{z^{ e_j}}
		&=\sum_{\beta\in\mathbb{N}_0^4}a_0^\beta\ip{z^{\beta+2 e_i}}{z^{ e_j}}+\sum_{\abs{\alpha}=1}\sum_{\beta\in\mathbb{N}_0^4}a^\beta_\alpha(\alpha\cdot(2 e_i))\ip{z^{\beta+2 e_i-\alpha}}{z^{ e_j}}
			=0, \nonumber \\
		\ip{z^{2 e_i}}{Lz^{ e_j}}
		&=\sum_{\beta\in\mathbb{N}_0^4}\overline{a^\beta_0}\overline{\ip{z^{\beta+ e_j}}{z^{2 e_i}}}+\sum_{\abs{\alpha}=1}\sum_{\beta\in\mathbb{N}_0^4}\overline{a^\beta_\alpha}(\alpha\cdot e_j)\overline{\ip{z^{\beta+ e_j-\alpha}}{z^{2 e_i}}} \nonumber \\
		&=\overline{a^{2 e_i}_{ e_j}}\norm{z^{2 e_i}}^2
			=\overline{a^{2 e_i}_{ e_j}}\frac{\pi^4(\xi+1)!}{(\xi+5)!}. \nonumber
	\end{align}
	which gives us $a^{2 e_j}_{ e_i}=\overline{a^{2 e_j}_{ e_i}}=0$.
	For $n=m=2 e_i$, we have
	\begin{align}
		\ip{Lz^{2 e_i}}{z^{2 e_i}}_\xi
		&=\sum_{\beta\in\mathbb{N}_0^4}a^\beta_0\ip{z^{\beta+2 e_i}}{z^{2 e_i}}_\xi+\sum_{\abs{\alpha}=1}\sum_{\beta\in\mathbb{N}_0^4}a^\beta_\alpha(\alpha\cdot(2 e_i))\ip{z^{\beta+2 e_i-\alpha}}{z^{2 e_i}}_\xi \nonumber \\
		&=a_0^0\norm{z^{2 e_i}}^2_\xi+2a^{ e_i}_{ e_i}\norm{z^{2 e_i}}^2_\xi
			=(a_0^0+2a^{ e_i}_{ e_i})\frac{\pi^4(\xi+1)!}{(\xi+5)!}, \nonumber \\
		\ip{z^{2 e_i}}{Lz^{2 e_i}}_\xi
		&=\sum_{\beta\in\mathbb{N}_0^4}\overline{a^\beta_0}\overline{\ip{z^{\beta+2 e_i}}{z^{2 e_i}}_\xi}+\sum_{\abs{\alpha}=1}\sum_{\beta\in\mathbb{N}_0^4}\overline{a^\beta_\alpha}(\alpha\cdot(2 e_i))\overline{\ip{z^{\beta+2 e_i-\alpha}}{z^{2 e_i}}_\xi} \nonumber \\
		&=\overline{a_0^0}\norm{z^{2 e_i}}_\xi+2\overline{a^{ e_i}_{ e_i}}\norm{z^{2 e_i}}^2_\xi
			=(\overline{a_0^0}+2\overline{a^{ e_i}_{ e_i}})\frac{\pi^4(\xi+1)!}{(\xi+5)!}, \nonumber
	\end{align}
	yielding no new relations.
	For $n=2 e_i$ and $m= e_i+ e_{5-i}$, we have
	\begin{align}
		\ip{Lz^{2 e_i}}{z^{ e_i+ e_{5-i}}}_\xi
		&=\sum_{\beta\in\mathbb{N}_0^4}a^\beta_0\ip{z^{\beta+2 e_i}}{z^{ e_i+ e_{5-i}}}_\xi+\sum_{\abs{\alpha}=1}\sum_{\beta\in\mathbb{N}_0^4}a^\beta_\alpha(\alpha\cdot(2 e_i))\ip{z^{\beta+2 e_i-\alpha}}{z^{ e_i+ e_{5-i}}}_\xi \nonumber \\
		&=2a_{ e_i}^{ e_{5-i}}\norm{z^{ e_i+ e_{5-i}}}^2_\xi 
			=2a^{ e_{5-i}}_{ e_i}\frac{\pi^4(\xi+1)!(\xi+4)}{6(\xi+5)!}, \nonumber \\
		\ip{z^{2 e_i}}{Lz^{ e_i+ e_{5-i}}}_\xi
		&=\sum_{\beta\in\mathbb{N}_0^4}\overline{a^\beta_0}\overline{\ip{z^{\beta+( e_i+ e_{5-i})}}{z^{2 e_i}}_\xi} \nonumber \\
		&\qquad+\sum_{\abs{\alpha}=1}\sum_{\beta\in\mathbb{N}_0^4}\overline{a^\beta_\alpha}(\alpha\cdot( e_i+ e_{5-i}))\overline{\ip{z^{\beta+( e_i+ e_{5-i})-\alpha}}{z^{2 e_i}}_\xi} \nonumber \\
		&=\overline{a_{ e_{5-i}}^{ e_i}}\norm{z^{2 e_i}}^2_\xi
			=\overline{a^{ e_i}_{ e_{5-i}}}\frac{\pi^4(\xi+1)!}{(\xi+5)!}, \nonumber
	\end{align}
	which gives us $a^{ e_{5-i}}_{ e_i}=\overline{a^{ e_i}_{ e_{5-i}}}\frac{3}{\xi+4}$, and combined with the relation $a^{ e_i}_{ e_j}=\overline{a^{ e_j}_{ e_i}}$, gives us that $a^{ e_{5-i}}_{ e_i}=\overline{a^{ e_i}_{ e_{5-i}}}=0$.
	For $n=2 e_i$ and $m= e_i+ e_j$, we have
	\begin{align}
		\ip{Lz^{2 e_i}}{z^{ e_i+ e_j}}_\xi
		&=\sum_{\beta\in\mathbb{N}_0^4}a^\beta_0\ip{z^{\beta+2 e_i}}{z^{ e_i+ e_j}}_\xi \nonumber \\
		&\qquad+\sum_{\abs{\alpha}=1}\sum_{\beta\in\mathbb{N}_0^4}a^\beta_\alpha(\alpha\cdot(2 e_i))\ip{z^{\beta+2 e_i-\alpha}}{z^{ e_i+ e_j}}_\xi \nonumber \\
		&=2a_{ e_i}^{ e_j}\norm{z^{ e_i+ e_j}}^2_\xi
			=a^{ e_j}_{ e_i}\frac{\pi^4(\xi+1)!}{(\xi+5)!}, \nonumber \\
		\ip{z^{2 e_i}}{Lz^{ e_i+ e_j}}_\xi
		&=\sum_{\beta\in\mathbb{N}_0^4}\overline{a^\beta_0}\overline{\ip{z^{\beta+( e_i+ e_j)}}{z^{2 e_i}}_\xi} \nonumber \\
		&\qquad+\sum_{\abs{\alpha}=1}\sum_{\beta\in\mathbb{N}_0^4}\overline{a^\beta_\alpha}(\alpha\cdot( e_i+ e_j))\overline{\ip{z^{\beta+( e_i+ e_j)-\alpha}}{z^{2 e_i}}_\xi} \nonumber \\
		&=\overline{a_{ e_j}^{ e_i}}\norm{z^{2 e_i}}^2_\xi
			=\overline{a^{ e_i}_{ e_j}}\frac{\pi^4(\xi+1)!}{(\xi+5)!}, \nonumber
	\end{align}
	yielding no new relations.
	For $n=2 e_i$ and $m= e_j+ e_{5-j}$, we have
	\begin{align}
		\ip{Lz^{2 e_i}}{z^{ e_j+ e_{5-j}}}_\xi
		&=\sum_{\beta\in\mathbb{N}_0^4}a^\beta_0\ip{z^{\beta+2 e_i}}{z^{ e_j+ e_{5-j}}}_\xi \nonumber \\
		&\qquad+\sum_{\abs{\alpha}=1}\sum_{\beta\in\mathbb{N}_0^4}a^\beta_\alpha(\alpha\cdot(2 e_i))\ip{z^{\beta+2 e_i-\alpha}}{z^{ e_j+ e_{5-j}}}_\xi \nonumber \\
		&=2a_{ e_i}^{ e_{5-i}}\ip{z^{ e_i+ e_{5-i}}}{z^{ e_j+ e_{5-j}}}_\xi
			=a^{ e_{5-i}}_{ e_i}\frac{\pi^4(\xi+1)!}{3(\xi+5)!}, \nonumber \\
		\ip{z^{2 e_i}}{Lz^{ e_j+ e_{5-j}}}_\xi
		&=\sum_{\beta\in\mathbb{N}_0^4}\overline{a^\beta_0}\overline{\ip{z^{\beta+( e_j+ e_{5-j})}}{z^{2 e_i}}_\xi} \nonumber \\
		&\qquad+\sum_{\abs{\alpha}=1}\sum_{\beta\in\mathbb{N}_0^4}\overline{a^\beta_\alpha}(\alpha\cdot( e_j+ e_{5-j}))\overline{\ip{z^{\beta+( e_j+ e_{5-j})-\alpha}}{z^{2 e_i}}_\xi}
			=0, \nonumber
	\end{align}
	which yields no new relations.
	For $n= e_i+ e_{5-i}$ and $m= e_i$,
	\begin{align}
		\ip{Lz^{ e_i+ e_{5-i}}}{z^{ e_i}}_\xi
		&=\sum_{\beta\in\mathbb{N}_0^4}a^\beta_0\ip{z^{\beta+(\gamma+ e_{5-i})}}{z^{ e_i}}_\xi \nonumber \\
		&\qquad+\sum_{\abs{\alpha}=1}\sum_{\beta\in\mathbb{N}_0^4}a^\beta_\alpha(\alpha\cdot( e_i+ e_{5-i}))\ip{z^{\beta+( e_i+ e_{5-i})-\alpha}}{z^{ e_i}}_\xi \nonumber \\
		&=a^0_{ e_{5-i}}\norm{z^{ e_i}}^2_\xi
			=a^0_{ e_{5-i}}\frac{\pi^4(\xi+1)!}{2(\xi+4)!}, \nonumber \\
		\ip{z^{ e_i+ e_{5-i}}}{Lz^{ e_i}}_\xi
		&=\sum_{\beta\in\mathbb{N}_0^4}\overline{a^\beta_0}\overline{\ip{z^{\beta+ e_i}}{z^{ e_i+ e_{5-i}}}_\xi} \nonumber \\
		&\qquad+\sum_{\abs{\alpha}=1}\sum_{\beta\in\mathbb{N}_0^4}\overline{a^\beta_\alpha}(\alpha\cdot e_i)\overline{\ip{z^{\beta+ e_i-\alpha}}{z^{ e_i+ e_{5-i}}}_\xi} \nonumber \\
		&=\overline{a^{ e_{5-i}}_0}\norm{z^{ e_i+ e_{5-i}}}^2_\xi \nonumber \\
		&\qquad+\overline{a^{ e_i+ e_{5-i}}_{ e_i}}\norm{z^{ e_i+ e_{5-i}}}^2_\xi+\overline{a^{ e_j+ e_{5-j}}_{ e_i}}\overline{\ip{z^{ e_j+ e_{5-j}}}{z^{ e_i+ e_{5-i}}}_\xi} \nonumber \\
		&=\overline{a^{ e_{5-i}}_0}\frac{\pi^4(\xi+1)!(\xi+4)}{6(\xi+5)!} \nonumber \\
		&\qquad+\overline{a^{ e_i+ e_{5-i}}_{ e_i}}\frac{\pi^4(\xi+1)!(\xi+4)}{6(\xi+5)!}-\overline{a^{ e_j+ e_{5-j}}_{ e_i}}\frac{\pi^4(\xi+1)!}{6(\xi+5)!}, \nonumber
	\end{align}
	thus giving us $a^{ e_j+ e_{5-j}}_{ e_i}=\overline{a^0_{ e_{5-i}}}+(\xi+4)a^{ e_i+ e_{5-i}}_{ e_i}=\overline{a^0_{ e_{5-i}}}$.
	For $n= e_i+ e_{5-i}$ and $m= e_j$,
	\begin{align}
		\ip{Lz^{ e_i+ e_{5-i}}}{z^{ e_j}}_\xi
		&=\sum_{\beta\in\mathbb{N}_0^4}a^\beta_0\ip{z^{\beta+(\gamma+ e_{5-i})}}{z^{ e_j}}_\xi \nonumber \\
		&\qquad+\sum_{\abs{\alpha}=1}\sum_{\beta\in\mathbb{N}_0^4}a^\beta_\alpha(\alpha\cdot( e_i+ e_{5-i}))\ip{z^{\beta+( e_i+ e_{5-i})-\alpha}}{z^{ e_j}}_\xi
			=0, \nonumber \\
		\ip{z^{ e_i+ e_{5-i}}}{Lz^{ e_j}}_\xi
		&=\sum_{\beta\in\mathbb{N}_0^4}\overline{a^\beta_0}\overline{\ip{z^{\beta+ e_j}}{z^{ e_i+ e_{5-i}}}_\xi} \nonumber \\
		&\qquad+\sum_{\abs{\alpha}=1}\sum_{\beta\in\mathbb{N}_0^4}\overline{a^\beta_\alpha}(\alpha\cdot e_j)\overline{\ip{z^{\beta+ e_j-\alpha}}{z^{ e_i+ e_{5-i}}}_\xi} \nonumber \\
		&=\overline{a^{ e_{5-j}}_0}\overline{\ip{z^{ e_j+ e_{5-j}}}{z^{ e_i+ e_{5-i}}}_\xi} \nonumber \\
		&\qquad+\overline{a^{ e_i+ e_{5-i}}_{ e_j}}\norm{z^{ e_i+ e_{5-i}}}^2_\xi+\overline{a^{ e_j+ e_{5-j}}_{ e_j}}\overline{\ip{z^{ e_j+ e_{5-j}}}{z^{ e_i+ e_{5-i}}}_\xi} \nonumber \\
		&=-\overline{a^{ e_{5-j}}_0}\frac{\pi^4(\xi+1)!}{6(\xi+5)!} \nonumber \\
		&\qquad+\overline{a^{ e_i+ e_{5-i}}_{ e_j}}\frac{\pi^4(\xi+1)!(\xi+4)}{6(\xi+5)!}-\overline{a^{ e_j+ e_{5-j}}_{ e_j}}\frac{\pi^4(\xi+1)!}{6(\xi+5)!}, \nonumber 
	\end{align}
	giving us $a^{ e_j+ e_{5-j}}_{ e_i}=a^{ e_{5-i}}_0\frac{1}{\xi+4}+a^{ e_i+ e_{5-i}}_{ e_i}\frac{1}{\xi+4}=\overline{a^0_{ e_{5-i}}}+a^{ e_i+ e_{5-i}}_{ e_i}\frac{1}{\xi+4}$.
	Combining with previous relations, we get that $a^{ e_i+ e_{5-i}}_{ e_i}=0$, which gives the relation $a^{ e_j+ e_{5-j}}_{ e_i}=\overline{a^0_{ e_{5-i}}}$.
	For $n=m= e_i+ e_{5-i}$, we have
	\begin{align}
		\ip{Lz^{ e_i+ e_{5-i}}}{z^{ e_i+ e_{5-i}}}_\xi
		&=\sum_{\beta\in\mathbb{N}_0^4}a^\beta_0\ip{z^{\beta+( e_i+ e_{5-i})}}{z^{ e_i+ e_{5-i}}}_\xi \nonumber \\
		&\quad+\sum_{\abs{\alpha}=1}\sum_{\beta\in\mathbb{N}_0^4}a^\beta_\alpha(\alpha\cdot( e_i+ e_{5-i}))\ip{z^{\beta+( e_i+ e_{5-i})-\alpha}}{z^{ e_i+ e_{5-i}}}_\xi \nonumber \\
		&=a_0^0\norm{z^{ e_i+ e_{5-i}}}^2_\xi+a^{ e_i}_{ e_i}\norm{z^{ e_i+ e_{5-i}}}^2_\xi+a^{ e_{5-i}}_{ e_{5-i}}\norm{z^{ e_i+ e_{5-i}}}^2_\xi \nonumber \\
		&=(a_0^0+a^{ e_i}_{ e_i}+a^{ e_{5-i}}_{ e_{5-i}})\frac{\pi^4(\xi+1)!(\xi+4)}{6(\xi+5)!}, \nonumber \\
		\ip{z^{ e_i+ e_{5-i}}}{Lz^{ e_i+ e_{5-i}}}_\xi
		&=\sum_{\beta\in\mathbb{N}_0^4}\overline{a^\beta_0}\overline{\ip{z^{\beta+( e_i+ e_{5-i})}}_\xi} \nonumber \\
		&\quad+\sum_{\abs{\alpha}=1}\sum_{\beta\in\mathbb{N}_0^4}\overline{a^\beta_\alpha}(\alpha\cdot( e_i+ e_{5-i}))\overline{\ip{z^{\beta+( e_i+ e_{5-i})-\alpha}}{z^{ e_i+ e_{5-i}}}_\xi} \nonumber \\
		&=\overline{a_0^0}\norm{z^{ e_i+ e_{5-i}}}_\xi+\overline{a^{ e_i}_{ e_i}}\norm{z^{ e_i+ e_{5-i}}}_\xi+\overline{a^{ e_{5-i}}_{ e_{5-i}}}\norm{z^{ e_i+ e_{5-i}}}_\xi \nonumber \\
		&\qquad=(\overline{a_0^0}+\overline{a^{ e_i}_{ e_i}}+\overline{a^{ e_{5-i}}_{ e_{5-i}}})\frac{\pi^4(\xi+1)!(\xi+4)}{6(\xi+5)!}, \nonumber
	\end{align} 
	which yields no new relations.
	For $n= e_i+ e_{5-i}$ and $m= e_i+ e_j$,
	\begin{align}
		\ip{Lz^{ e_i+ e_{5-i}}}{z^{ e_i+ e_j}}_\xi
		&=\sum_{\beta\in\mathbb{N}_0^4}a^\beta_0\ip{z^{\beta+( e_i+ e_{5-i})}}{z^{ e_i+ e_j}}_\xi \nonumber \\
		&\quad+\sum_{\abs{\alpha}=1}\sum_{\beta\in\mathbb{N}_0^4}a^\beta_\alpha(\alpha\cdot( e_i+ e_{5-i}))\ip{z^{\beta+( e_i+ e_{5-i})-\alpha}}{z^{ e_i+ e_j}}_\xi \nonumber \\
		&=a^{ e_j}_{ e_{5-i}}\norm{z^{ e_i+ e_j}}_\xi
			=a^{ e_j}_{ e_{5-i}}\frac{\pi^4(\xi+1)!}{2(\xi+5)!}, \nonumber \\
		\ip{z^{ e_i+ e_{5-i}}}{Lz^{ e_i+ e_j}}_\xi
		&=\sum_{\beta\in\mathbb{N}_0^4}\overline{a^{\beta}_0}\overline{\ip{z^{\beta+( e_i+ e_j)}}{z^{ e_i+ e_{5-i}}}_\xi} \nonumber \\
		&\quad+\sum_{\abs{\alpha}=1}\sum_{\beta\in\mathbb{N}_0^4}\overline{a^\beta_\alpha}(\alpha\cdot( e_i+ e_j))\overline{\ip{z^{\beta+( e_i+ e_j)-\alpha}}{z^{ e_i+ e_{5-i}}}_\xi} \nonumber \\
		&=\overline{a^{ e_{5-i}}_{ e_j}}\norm{z^{ e_i+ e_{5-i}}}^2_\xi+\overline{a^{ e_{5-j}}_{ e_i}}\overline{\ip{z^{ e_j+ e_{5-j}}}{z^{ e_i+ e_{5-i}}}_\xi} \nonumber \\
		&\qquad=\overline{a^{ e_{5-i}}_{ e_j}}\frac{\pi^4(\xi+1)!(\xi+4)}{6(\xi+5)!}-\overline{a^{ e_{5-j}}_{ e_i}}\frac{\pi^4(\xi+1)!}{6(\xi+5)!}, \nonumber
	\end{align}
		which gives us the relation $a^{ e_j}_{ e_i}=a^{ e_{5-i}}_{ e_{5-j}}$. 
	For $n= e_i+ e_{5-i}$ and $m= e_j+ e_{5-j}$, we have
	\begin{align}
		\ip{Lz^{ e_i+ e_{5-i}}}{z^{ e_j+ e_{5-j}}}_\xi
		&=\sum_{\beta\in\mathbb{N}_0^4}a^\beta_0\ip{z^{\beta+( e_i+ e_{5-i})}}{z^{ e_j+ e_{5-j}}}_\xi \nonumber \\
		&\quad+\sum_{\abs{\alpha}=1}\sum_{\beta\in\mathbb{N}_0^4}a^\beta_\alpha(\alpha\cdot( e_i+ e_{5-i}))\ip{z^{\beta+( e_i+ e_{5-i})-\alpha}}{z^{ e_j+ e_{5-j}}}_\xi \nonumber \\
		&=a_0^0\ip{z^{ e_i+ e_{5-i}}}{z^{ e_j+ e_{5-j}}}_\xi+a^{ e_i}_{ e_i}\ip{z^{ e_i+ e_{5-i}}}{z^{ e_j+ e_{5-j}}}_\xi \nonumber \\
		&\qquad+a^{ e_{5-i}}_{ e_{5-i}}\ip{z^{ e_i+ e_{5-i}}}{z^{ e_j+ e_{5-j}}}_\xi \nonumber \\
		&=-(a_0^0+a^{ e_i}_{ e_i}+a^{ e_{5-i}}_{ e_{5-i}})\frac{\pi^4(\xi+1)!}{6(\xi+5)!}, \nonumber \\
		\ip{z^{ e_i+ e_{5-i}}}{Lz^{ e_j+ e_{5-j}}}_\xi
		&=\sum_{\beta\in\mathbb{N}_0^4}\overline{a^\beta_0}\overline{\ip{z^{\beta+( e_j+ e_{5-j})}}{z^{ e_i+ e_{5-i}}}_\xi} \nonumber \\
		&\quad+\sum_{\abs{\alpha}=1}\sum_{\beta\in\mathbb{N}_0^4}\overline{a^\beta_\alpha}(\alpha\cdot( e_j+ e_{5-j}))\overline{\ip{z^{\beta+( e_j+ e_{5-j})-\alpha}}{z^{ e_i+ e_i'}}_\xi} \nonumber \\
		&=\overline{a_0^0}\overline{\ip{z^{ e_j+ e_{5-j}}}{z^{ e_i+ e_{5-i}}}_\xi}+\overline{a^{ e_j}_{ e_j}}\overline{\ip{z^{ e_j+ e_{5-j}}}{z^{ e_i+ e_{5-i}}}_\xi} \nonumber \\
		&\qquad+\overline{a^{ e_{5-j}}_{ e_{5-j}}}\overline{\ip{z^{ e_j+ e_{5-j}}}{z^{ e_i+ e_{5-i}}}_\xi} \nonumber \\
		&=-(\overline{a_0^0}+\overline{a^{ e_j}_{ e_j}}+\overline{a^{ e_{5-j}}_{ e_{5-j}}})\frac{\pi^4(\xi+1)!}{6(\xi+5)!}, \nonumber
	\end{align}
	which gives us that $a^{ e_i}_{ e_i}+a^{ e_{5-i}}_{ e_{5-i}}=a^{ e_j}_{ e_j}+a^{ e_{5-j}}_{ e_{5-j}}$.
	For $n= e_i+ e_j$ and $m= e_i$, we have
	\begin{align}
		\ip{Lz^{ e_i+ e_j}}{z^{ e_i}}_\xi
		&=\sum_{\beta\in\mathbb{N}_0^4}a^\beta_0\ip{z^{\beta+( e_i+ e_j)}}{z^{ e_i}}_\xi \nonumber \\
		&\quad+\sum_{\abs{\alpha}=1}\sum_{\beta\in\mathbb{N}_0^4}a^\beta_\alpha(\alpha\cdot( e_i+ e_j))\ip{z^{\beta+( e_i+ e_j)-\alpha}}{z^{ e_i}}_\xi \nonumber \\
		&=a^0_{ e_j}\norm{z^{ e_i}}^2_\xi
			=a^0_{ e_j}\frac{\pi^4(\xi+1)!}{2(\xi+4)!}, \nonumber \\
		\ip{z^{ e_i+ e_j}}{Lz^{ e_i}}_\xi
		&=\sum_{\beta\in\mathbb{N}_0^4}\overline{a^\beta_0}\overline{\ip{z^{\beta+ e_i}}{z^{ e_i+ e_j}}_\xi}+\sum_{\abs{\alpha}=1}\sum_{\beta\in\mathbb{N}_0^4}\overline{a^\beta_\alpha}(\alpha\cdot e_i)\overline{\ip{z^{\beta+ e_i-\alpha}}{z^{ e_i+ e_j}}_\xi} \nonumber \\
		&=\overline{a^{ e_j})_0}\norm{z^{ e_i+ e_j}}^2_\xi+\overline{a_{ e_i}^{ e_i+ e_j}}\norm{z^{ e_i+ e_j}}^2_\xi
			=(\overline{a^{ e_j}_0}+\overline{a_{ e_i}^{ e_i+ e_j}})\frac{\pi^4(\xi+1)!}{2(\xi+5)!}, \nonumber
	\end{align}
	thus $(\xi+5)a^0_{ e_j}=\overline{a^{ e_j}_0}+\overline{a^{ e_i+ e_j}_{ e_i}}$, and using previous relations, gives us $a^{ e_i+ e_j}_{ e_i}=\overline{a^0_{ e_j}}$.
	For $n= e_i+ e_j$ and $m= e_{5-i}$, 
	\begin{align}
		\ip{Lz^{ e_i+ e_j}}{z^{ e_{5-i}}}_\xi
		&=\sum_{\beta\in\mathbb{N}_0^4}a^\beta_0\ip{z^{\beta+( e_i+ e_j)}}{z^{ e_{5-i}}}_\xi \nonumber \\
		&\quad+\sum_{\abs{\alpha}=1}\sum_{\beta\in\mathbb{N}_0^4}a^\beta_\alpha(\alpha\cdot( e_i+ e_j))\ip{z^{\beta+( e_i+ e_j)-\alpha}}{z^{ e_{5-i}}}_\xi=0, \nonumber \\
		\ip{z^{ e_i+ e_j}}{Lz^{ e_{5-i}}}_\xi
		&=\sum_{\beta\in\mathbb{N}_0^4}\overline{a^\beta_0}\overline{\ip{z^{\beta+ e_{5-i}}}{z^{ e_i+ e_j}}_\xi} \nonumber \\
		&\quad+\sum_{\abs{\alpha}=1}\sum_{\beta\in\mathbb{N}_0^4}\overline{a^\beta_\alpha}(\alpha\cdot e_{5-i})\overline{\ip{z^{\beta+ e_{5-i}-\alpha}}{z^{ e_i+ e_j}}_\xi} \nonumber \\
		&=\overline{a_{ e_{5-i}}^{ e_i+ e_j}}\norm{z^{ e_i+ e_j}}^2_\xi
			=\overline{a^{ e_i+ e_j}_{ e_{5-i}}}\frac{\pi^4(\xi+1)!}{2(\xi+5)!}, \nonumber
	\end{align}
	which gives us $a^{ e_{5-i}+ e_j}_{ e_i}=\overline{a^{ e_{5-i}+ e_j}_{ e_i}}=0$.
	For $n= e_i+ e_j$ and $m= e_i+ e_{5-i}$, we have
	\begin{align}
		\ip{Lz^{ e_i+ e_j}}{z^{ e_i+ e_{5-i}}}_\xi
		&=\sum_{\beta\in\mathbb{N}_0^4}a^\beta_0\ip{z^{\beta+( e_i+ e_j)}}{z^{ e_i+ e_{5-i}}}_\xi \nonumber \\
		&\quad+\sum_{\abs{\alpha}=1}\sum_{\beta\in\mathbb{N}_0^4}a^\beta_\alpha(\alpha\cdot( e_i+ e_j))\ip{z^{\beta+( e_i+ e_j)-\alpha}}{z^{ e_i+ e_{5-i}}}_\xi \nonumber \\
		&=a^{ e_{5-i}}_{ e_j}\norm{z^{ e_i+ e_{5-i}}}^2_\xi+a^{ e_{5-j}}_{ e_i}\ip{z^{ e_j+ e_{5-j}}}{z^{ e_i+ e_{5-i}}}_\xi \nonumber \\
		&=a^{ e_{5-i}}_{ e_j}\frac{\pi^4(\xi+1)!(\xi+4)}{6(\xi+5)!}-a^{ e_{5-j}}_{ e_i}\frac{\pi^4(\xi+1)!}{6(\xi+5)!}, \nonumber \\
		\ip{z^{ e_i+ e_j}}{Lz^{ e_i+ e_{5-i}}}_\xi
		&=\sum_{\beta\in\mathbb{N}_0^4}\overline{a^\beta_0}\overline{\ip{z^{\beta+( e_i+ e_{5-i})}}{z^{ e_i+ e_j}}_\xi} \nonumber \\
		&\quad+\sum_{\abs{\alpha}=1}\sum_{\beta\in\mathbb{N}_0^4}\overline{a^\beta_\alpha}(\alpha\cdot( e_i+ e_{5-i}))\overline{\ip{z^{\beta+( e_i+ e_{5-i})-\alpha}}{z^{ e_i+ e_j}}_\xi} \nonumber \\
		&=\overline{a_{ e_{5-i}}^{ e_j}}\norm{z^{ e_i+ e_j}}^2_\xi
			=\overline{a^{ e_j}_{ e_{5-i}}}\frac{\pi^4(\xi+1)!}{2(\xi+5)!} \nonumber
	\end{align}
	which gives us no new relations.

\bibliographystyle{plain}
\bibliography{Bibtex}

\end{document}